\documentclass[10pt,twocolumn,twoside]{IEEEtran}

\IEEEoverridecommandlockouts                              

\usepackage{graphics} 
\usepackage{epsfig} 
\usepackage{amsmath} 

\usepackage{amssymb,amsthm}  
\usepackage{amscd}
\usepackage[noadjust]{cite}
\usepackage{float}
\usepackage{times}
\usepackage{color}
\usepackage[utf8]{inputenc}
\usepackage[english]{babel}
\usepackage{algorithm}
\usepackage{multicol}

\usepackage[noend]{algpseudocode}
\newtheorem{theorem}{Theorem}
\newtheorem{lemma}{Lemma}

\newtheorem{Remark}{Remark}
\newtheorem{proposition}{Proposition}
\newtheorem{corollary}{Corollary}
\newtheorem{definition}{Definition}

\newcommand{\F}{\mathbb{F}}
\newcommand{\N}{\mathcal{N}_{\rm out}}
\newcommand{\Nin}{\mathcal{N}_{\rm in}}
\newcommand{\Eout}{\mathcal{E}_{\rm out}}
\newcommand{\Ein}{\mathcal{E}_{\rm in}}

\title{\LARGE \bf Controllability of Conjunctive Boolean Networks with \\Application to Gene Regulation
}

\author{Zuguang Gao, Xudong Chen, Tamer Ba\c{s}ar
\thanks{This research was supported in part by the Office of Naval Research (ONR) MURI grant N-00014-16-1-2710.
Zuguang Gao and Tamer Ba\c{s}ar are with the Coordinated Science Laboratory, University of Illinois at Urbana-Champaign. Emails: \{zgao19, basar1\}@illinois.edu. Xudong Chen is with the Department of ECEE, University of Colorado Boulder. Email: xudong.chen@colorado.edu.} 
}

\begin{document}

\maketitle
\thispagestyle{empty}
\pagestyle{empty}

\begin{abstract}
	
A Boolean network is a finite state discrete time dynamical system. At each step, each variable takes a value from a binary set. The value update rule for each variable is a local function which depends only on a selected subset of variables. Boolean networks have been used in modeling gene regulatory networks. We focus in this paper on a special class of Boolean networks, namely the conjunctive Boolean networks (CBNs), whose value update rule is comprised of only logic AND operations. It is known that any trajectory of a Boolean network will enter a periodic orbit. Periodic orbits of a CBN have been completely understood. In this paper, we investigate the {\em orbit-controllability} and {\em state-controllability} of a CBN: We ask the question of how one can steer a CBN to enter any periodic orbit or to reach any final state, from any initial state. 
We establish necessary and sufficient conditions for a CBN to be orbit-controllable and state-controllable. Furthermore, explicit control laws are presented along the analysis.

\end{abstract}


\section{Introduction}

One of the central focuses of today's genomic research is to study the regulation of gene expressions, i.e., the underlying mechanism used by a cell to execute and control the production of gene products (protein or RNA)~\cite{shmulevich2002boolean}. Questions about how to model such a mechanism become more and more relevant and have been studied to some extent. In particular, we note here two different approaches for modeling the interactions among the genes in a regulatory network---one is called the ``dynamic-system" method and the other is called the ``Boolean" method~\cite{smolen2000mathematical}. Specifically, the dynamic-system method uses ordinary differential equations to describe the rates of change of the concentrations of gene products. 
Yet, the associated differential equations are often quite complex and do not admit explicit solutions. For large-sized gene networks, computer simulation of the evolution of the dynamics usually takes a significant amount of time. The Boolean method, on the other hand, leads to some loss of accuracy due to simplifying the expression status of a gene to  a Boolean variable. Such a simplification, however, makes it possible to analyze and simulate the interactions among genes, and hence finds several natural applications (see, for example, $\lambda$-bacteriophage circuitry~\cite{hasty2001computational}). Our focus in this paper will be on the Boolean method.

Since the expression process of a gene involves participation of proteins, which are products of some other genes, genes interact with each other through their products~\cite{noual2013non}. These interactions can then be naturally described by certain types of Boolean functions whose inputs are the previous values of the  genes and the outputs are their updated values. Boolean variables, combined with Boolean functions comprise a Boolean network, which is a discrete-time dynamical system with a finite state space (finite dynamical system). 
Boolean networks were originally introduced in~\cite{kauffman1969metabolic,kauffman1969homeostasis}, later generalized in~\cite{thomas1990biological}, and have been extensively used in systems biology and (mathematical) computational biology~\cite{mcculloch1943logical,hopfield1982neural,hopfield1984neurons,akutsu1999identification,davidich2008boolean}.

{\em Boolean functions.} 
There have been extensive studies of various classes of Boolean functions which are particularly suited to the logical expression of gene regulation \cite{thomas1973boolean,raeymaekers2002dynamics}. 
Evidence has been provided in~\cite{sontag2008effect} that biochemical networks are ``close to monotone". Roughly speaking, a Boolean network is monotonic if its Boolean function has the property that the output value of the function for each variable is non-decreasing if the number of ``$1$"s in the inputs increases. 
{\color{black}
For example, Boolean networks whose Boolean functions are monomials~\cite{colon2005boolean,colon2006monomial,park2014monomial,jarrah2010dynamics} are monotonic. 
For other types of monotonic Boolean networks, we refer the reader to~\cite{aracena2004limit,zhao2005remark,melliti2016asynchronous,lingas2017towards} and the references therein. }
{\color{black} 
A special type of monotonic Boolean functions, of particular interest to us, is
those comprised of only AND operations. 
The corresponding Boolean networks are said to be {\em conjunctive}~\cite{jarrah2010dynamics}. 

{\color{black}
Conjunctive Boolean networks (CBNs) constitute an appealing model in systems biology, especially in the study of gene regulation. A gene is a portion of the DNA, and in the expression process of a gene, the DNA is first transcribed to mRNA, which is then translated to one or several proteins, called the product of that gene. Since proteins can influence the transcription and translation stages, genes interact with each other through their products. {\color{black}In a CBN, the status of each gene is either ``on'' or ``off'', indicating whether it is expressed or not, and is represented by the Boolean variable ``1'' or ``0''.  Now, consider the situation where the expression process of a gene involves the participation of several proteins, and these proteins can be produced by a selected subset of genes in the network during the previous time step. Then, this gene is expressed if and only if all the genes in the selected subset were expressed in the previous time step. Therefore, the dynamics of a CBN captures a certain aspect of the interactions among the genes while entailing a tractable analysis.}
We further refer to Fig.~\ref{gene} (originally from~\cite{translation} and reproduced here) for the validity of CBNs in modeling the process of gene expressions.
}

\begin{figure}[h]
	\centering
	\includegraphics[height=77mm]{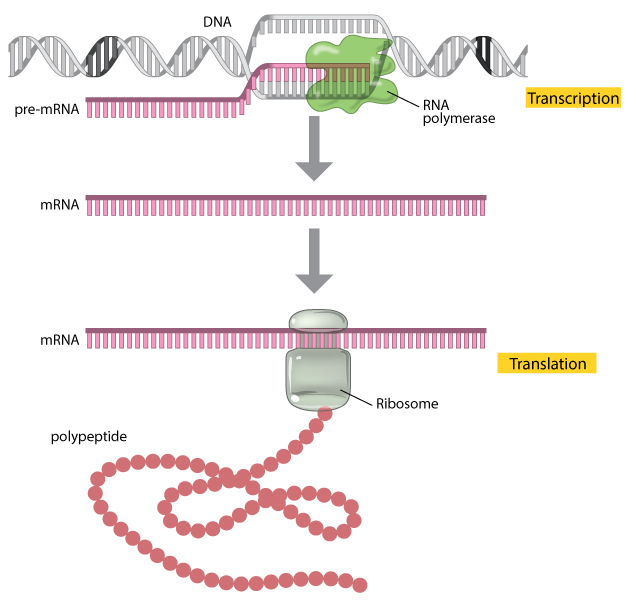}
	\caption{This figure, originally from~\cite{translation}, illustrates the expression process of a gene. It can be seen that the transcription stage requires the participation of RNA polymerase, which is essentially a protein. The translation stage involves ribosome, which contains ribosomal proteins. These proteins are all products of some other genes at the previous time steps. Thus, the gene in the figure can be expressed (holding ``$1$") if and only if all other related genes were expressed (holding ``$1$") previously. 
	}
	\label{gene}
\end{figure}

}

{\em Problem description.} 
We address in the paper the controllability of a CBN. Assuming that there is a subset of variables whose values are determined by external inputs (the controls), we ask and answer two questions. First, how can one steer the system from any initial state to any desired periodic orbit\footnote{A CBN is a finite dynamical system. Thus, for any initial condition, the trajectory generated by the system will enter a periodic orbit (also known as a limit cycle) in finite time steps.}? If this is possible, we say that the system is \emph{orbit-controllable} and the subset of variables  whose values are determined by external inputs (the controls) is termed \emph{orbit-controlling set}. Second, how can one make the system \emph{state-controllable}, meaning that the trajectory generated by the control system to be driven into any desired final state (not necessarily a state in a periodic orbit), starting from any initial condition? When the system is state-controllable, the subset of variables is termed the \emph{state-controlling set}. Note that state-controllability is a stronger notion than orbit-controllability, and hence it is more restrictive  for a subset to be a state-controlling set than to be an orbit-controlling set. {\color{black} As mentioned earlier, the control problems posed here find their applications in gene regulation, where the objective is to control the expressions of a selected subset of genes so as to steer a bio-system to reach a desired final state (or a periodic orbit)~\cite{gossen1992tight,derisi1997exploring,stragier1988processing,helene1991anti,menolascina2014vivo,milias2011silico,pathak2013optogenetic,uhlendorf2012long}, and hence to look for criteria for the selection so that the system is controllable. }

{\color{black} Reachability and observability for general Boolean networks have been addressed to some extent~\cite{
		cheng2009controllability,
		zhang2012controllability,
		liu2014some,
		liu2016state,
		luo2016controllability}. For example, \cite{cheng2010analysis} used a semi-tensor product approach to establish necessary and sufficient conditions for a given final state to be reachable from a given initial state; \cite{laschov2012controllability} also addressed the reachability question, but via the Perron$-$Frobenius theory; 
	\cite{li2015controllability} studied the controllability (as well as observability) of a Boolean network by looking at the algebraic variety of a certain ideal generated by certain polynomials defined over the finite field $\mathbb{F}_2 = \{0,1\}$. 
We adopt, in this paper, a graphical approach to address the controllability question which, to the best of our knowledge, is different from all the other existing methods, thus providing a new perspective. We provide necessary and sufficient conditions for a subset of variables to be an orbit-controlling set (Theorem~\ref{thm1}) and a state-controlling set (Theorem~\ref{thm2}).  Furthermore, explicit control laws for steering the system to a desired periodic orbit or desired final state are also provided.
While the ultimate goal is to find an orbit- or state-controlling set with minimal cardinality, the condition we establish in this paper helps reduce the size of such a set significantly. 
}

This paper is based on some preliminary results of two conference papers~\cite{orbitcontrol,statecontrol}. Specifically, this paper provides full details of analyses, proofs and examples, some of which were left out in the conference versions due to space limitation. While we have borrowed some results from~\cite{stabilityfull}  for the proof of Theorem~\ref{thm1}, the problem we
are solving in this paper differs from~\cite{stabilityfull} in the sense that
we have added controls to the network, which did not exist
in the previous work. The rest of the paper is organized as follows.
In Section~\ref{pre}, we first provide key definitions and notations  for directed graphs, binary necklace, and CBNs. We then formulate the controllability problem. In particular, we raise a two-part controllability question that is answered fully in the paper and introduce important related concepts. In Section~\ref{orbit} and Section~\ref{state}, we establish necessary and sufficient conditions for a CBN to be orbit-controllable (Theorem~\ref{thm1}) and state-controllable (Theorem~\ref{thm2}). The control procedures are also provided in Algorithms~\ref{orbitalgo} and~\ref{statealgo}. In the conclusions section, we summarize the main results of the paper and point out future research directions. The paper ends with an Appendix which contains analyses and proofs that are used to support a technical result.


\section{Preliminaries and Problem Formulation}\label{pre}

\subsection{Preliminaries}
{\em 1). Directed graph.}
We introduce here some notations associated with a directed graph (or simply digraph). Let $D=(V,E)$ be a directed graph, with $V$ the set of nodes (vertices) and $E$ the set of edges. We denote by $v_iv_j$ an edge from $v_i$ to $v_j$ in $D$. We say that $v_i$ is an {\em in-neighbor} of  $v_j$ and $v_j$ is an {\em out-neighbor} of $v_i$. 
The sets of in-neighbors and out-neighbors of node $v_i$ are denoted by $\mathcal{N}_{\rm in}(v_i)$ and
$\mathcal{N}_{{\rm out}}(v_i)$, respectively.  We write, on occasions,  $\mathcal{N}_{\rm in}(v_i; D)$ (resp. $\mathcal{N}_{\rm out}(v_i; D)$) to indicate that the in-neighbors (resp. out-neighbors) of $v_i$ are taken within the digraph $D$.
The  \emph{in-degree} and \emph{out-degree}
of node $v_i$ are defined to be $|\mathcal{N}_{{\rm in}}(v_i)|$ and $|\mathcal{N}_{{\rm out}}(v_i)|$, respectively.
We call $v_iv_j$ an \emph{out-edge} of $v_i$ and an \emph{in-edge} of $v_j$. We denote by $\Ein(v_i)$ (resp. $\Eout(v_i)$) the set of in-edges (resp. out-edges) of node $v_i$. 

Given a node $v_i$ of $V$ and a nonnegative integer $k$, we define a subset $\N^k(v_i)$ by induction: For $k = 0$, let $\N^0(v_i) := \{v_i\}$; for $k \ge 1$, we define 
\begin{equation}\label{eq:defninp2}
\N^k(v_i) := \cup_{v_j\in \N^{k-1}(v_i)} \N(v_j). 
\end{equation}
Note that if $\N^{k-1}(v_i)=\varnothing$, then $\N^k(v_i)=\varnothing$. Similarly, we define $\N^k(v_i)$ by replacing $\Nin$ with $\N$ in~\eqref{eq:defninp2}. 

Let $v_i$ and $v_j$ be two nodes of $D$. A {\em walk} from $v_i$ to $v_j$, denoted by $w_{ij}$, is a sequence $v_{i_0}v_{i_2}\cdots v_{i_m}$ (with $v_{i_0} = v_i$ and $v_{i_m} = v_j$) in which $v_{i_k}v_{i_{k+1}}$ is an edge of $D$ for all 
$k\in\{0,1,\ldots,m-1\}$. 
A walk is said to be a {\it path}, denoted by $p_{ij}$, if all the nodes in the walk are pairwise distinct. We use $P_{ij}$ to denote the set of all paths from $v_i$ to $v_j$. 
A walk is said to be a {\em cycle} if there is no repetition of nodes in the walk other than the repetition of the starting- and ending-node. The {\it length} of a path/cycle/walk is defined to be the number of edges in that path/cycle/walk. The length of a walk $w$ is denoted by $l(w)$, and the length of a path $p$ is denoted by $l(p)$.

A \emph{strongly connected graph} is a directed graph such that for any two nodes $v_i$ and $v_j$ in the graph, there is a path from $v_i$ to $v_j$. A \emph{directed acyclic graph (DAG)} is a directed graph containing no cycles. In a directed acyclic graph, a node with no in-neighbors (and hence no in-edges) is called a \emph{source node}. We note that in a DAG, any walk must also be a path.
For any digraph $D = (V, E)$, a \emph{subgraph} of $D=(V,E)$ is a digraph whose node set and edge set are subsets of $V$ and  $E$, respectively.

{\em 2). Binary necklace.}\label{binarynecklace}
A {\bf binary necklace} of length $p$ is an equivalence class of $p$-bead strings over the binary set $\mathbb{F}_2=\{0,1\}$, taking all rotations as equivalent. For example, {\color{black}in the case $p = 4$}, there are six different binary necklaces, as illustrated in Fig.~\ref{necklace}. 

\begin{figure}[h]
	\centering
	\includegraphics[height=40mm]{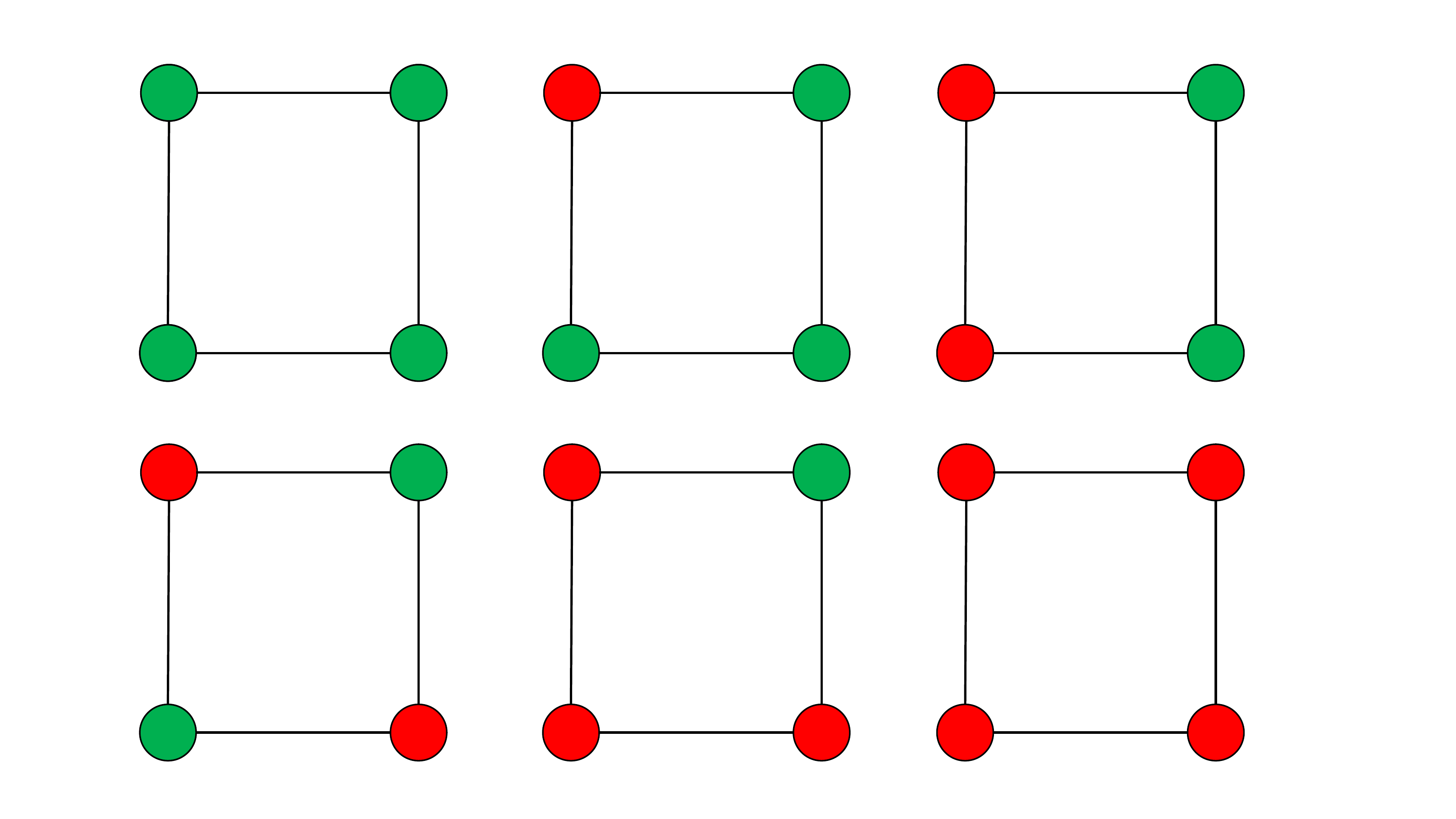}
	\caption{We illustrate here all binary necklaces of length~$4$. In the figure, if the bead is plotted in red (resp. green), then it holds value ``1'' (resp, ``0"). }
	\label{necklace}
\end{figure}

{\em 3). Conjunctive Boolean network} (CBN).
Let $\mathbb{F}_2:=\{0,1\}$ denote the finite field. 
A \emph{Boolean network} (BN) on $n$ Boolean variables $x_1(t),\ldots, x_n(t)\in \F_2$ is a discrete-time finite state dynamical system, whose update rule can be described by a set of Boolean functions $f_1,\ldots, f_n:$ 
$$
x_i(t+1) = f_i(x_1(t),\ldots, x_n(t)), \hspace{5pt} \forall i = 1, \ldots, n.  
$$
Let $x(t):= (x_1(t),\ldots, x_n(t)) \in \F_2^n$ be the state of the BN at time~$t$. Further, let $$f:= (f_1, \ldots, f_n): x(t) \mapsto x(t+1).$$ We refer to $f$ as the {\bf value update rule} associated with the BN. Note that in the sequel, all the Boolean variables are updated synchronously (in parallel) at each time step. For asynchronous (sequential) value updates,  we refer the reader to \cite{goles2012disjunctive,noual2013non,ruz2013preservation} for details.

Since a BN is a finite dynamical system, it is well known that for any initial condition $x(0)\in \F_2$, the trajectory $x(0), x(1), \ldots$ will enter a {\em periodic orbit} in a finite amount time. More precisely, there exists a time $t_0 \ge 0$ and an integer number $p \ge 1$ such that $x(t_0 + p) = x(t_0)$. Moreover, if $x(t_0 +q) \neq x(t_0)$ for any $q = 1,\ldots, p-1$, then the sequence $\{x(t_0), \ldots, x(t_0 + p-1)\}$, taking rotations as equivalent, is said to be a {\bf periodic orbit}, and $p$ is its {\bf period}. If the period of a periodic orbit is one, i.e., $x(t_0) = x(t_0 + k)$ for any $k \ge 1$, then the state $x(t_0)$ is said to be a {\bf fixed point}. 
{\color{black}We refer the reader to~\cite{aracena2016number,robert2012discrete} for studies on the number of fixed points of a BN.}
   
 We now introduce the following definition: 


\begin{definition}[Conjunctive Boolean network~\cite{jarrah2010dynamics}]\label{CBN}
A Boolean network $f = (f_1, \ldots, f_n)$  is {\bf conjunctive} if each Boolean function $f_i$, for all $i = 1,\ldots, n$, can be expressed as follows:
\begin{equation}\label{eq:updaterule}
f_i(x_1,\ldots, x_n)=\prod^n_{j = 1} x_j^{\epsilon_{ji}}
\end{equation}
with $\epsilon_{ji}\in \{0,1\}$ for all $j=1,\ldots,n$. 
\end{definition}


{\color{black}Note that states $(0,\ldots,0)$ and $(1,\ldots,1)$ are always fixed points for CBNs.}
We can associate with each CBN a unique directed graph, termed \emph{dependency graph}, whose definition is given below:

\begin{definition}[Dependency graph~\cite{jarrah2010dynamics}]\label{dependency}
Let $f = (f_1, \ldots, f_n)$ be the value update rule associated with a CBN. The associated {\bf dependency graph} is a directed graph $D = (V,E)$ of $n$ vertices. An edge from $v_i$ to $v_j$, denoted by $v_iv_j$,  exists in $E$ if $\epsilon_{ij} = 1$.   
\end{definition}



\subsection{Problem Formulation}\label{formulation}

In this section, we formally introduce the problem of how it would be possible to control a CBN. Specifically, we assume that there is a selected subset of nodes whose Boolean values can be controlled at any time.  We address in the paper the following controllability question:   
\begin{enumerate}
\item[{\it Q:}] {\it How can one steer a CBN from any initial state to any final state (or any periodic orbit) by controlling the values of the selected nodes?}
\end{enumerate}
We provide a complete answer to this question toward the end of the paper.



To proceed, we first introduce the control model in precise terms.  Let $D = (V, E)$ be the dependency graph of  a CBN. A node $v_i$ of $D$ is said to be a {\bf control node} if its value at any time step is determined completely by an external  control input. We denote by $V^*$ the subset of $V$, comprised of all the control nodes in the network. Then, the control model can be described as follows:   
\begin{equation}\label{eq:controlmodel}
x_i(t ) = 
\left \{
\begin{array}{ll}
u_i(t) & \mbox{ if } v_i \in V^*, \\
f_i(x(t -1))  & \mbox{ otherwise,}
\end{array}
\right.
\end{equation}
where the $u_i(\cdot)$'s  are the external control inputs, and the $f_i$'s are the Boolean functions given by~\eqref{eq:updaterule}. 
{\color{black}
For example, if the $u_i$'s are constant, then~\eqref{eq:controlmodel} simply models the mutants in genetic networks (i.e., $u_i=0$ represents a knock out of gene $i$).
}
We now introduce the following definitions:

{\color{black}
\begin{definition}[Orbit-controlling set]\label{orbitset}
	A subset $V^*\subseteq V$ is an {\bf orbit-controlling} set for~\eqref{eq:updaterule} if for any initial condition $x\in\F_2^n$ and any periodic orbit $\mathcal{O}$ of system~\eqref{eq:updaterule}, there exists a time $T$ and a set of control laws $u_i(t)$, for $v_i\in V^*$ and $0\leq t\leq T$, such that the trajectory generated by system~\eqref{eq:controlmodel} with $x(0)=x$, reaches a state in $\mathcal{O}$ at $t=T$.
	
\end{definition}

\begin{definition}[State-controlling set]\label{stateset}
	A subset $V^*\subseteq V$ is a {\bf state-controlling set} for~\eqref{eq:updaterule} if for any initial condition $x$ and any final state $x^*$, there exists a time $T$ and a set of control laws $u_i(t)$ for $v_i\in V^*$ and $0\leq t\leq T$ such that the trajectory generated by system~\eqref{eq:controlmodel} with $x(0)=x$, reaches $x^*$ at $t=T$.
	
\end{definition}


Note that a state-controlling set is an orbit-controlling set, but the converse is not necessarily true.}
Also, note that a state-controlling set always exists as one can set $V^* = V$. In this case, each node is a control node, and if we let $u_i(0) = x^*_i$, for all $v_i\in V$, then $x(0) = x^*$. 
However, the cost of controlling every node in the network could be extremely high, especially when the size of the network is large. From the biological perspective, controlling all genes in a bio-system is generally not feasible. One thus looks for a proper subset $V^*$, with $|V^*|\ll |V|$, such that $V^*$ is an orbit-controlling set (resp. state-controlling set). We take in the paper the first step to solve such a minimal controllability problem by providing a necessary and sufficient condition for a set $V^*$ to be an orbit-controlling (resp., a state-controlling) set. 

We recall that for a node $v_i$, with $v_i \notin V^*$,  the value $x_i(t)$ depends on the values of its incoming neighbors at time~$(t-1)$:
\begin{eqnarray}\label{xtneighbor}
x_i(t) =  \prod_{v_j\in \mathcal{N}_{\rm in}(v_i)}x_j(t-1)
\end{eqnarray}
The in-edges of $v_i$ thus demonstrate the information flow at the node~$v_i$. On the other hand, if $v_i$ is a control node, then from the model~\eqref{eq:controlmodel}, the value $x_i(t)$, at any time $t$,  is determined completely by an external input, rather than the values of its incoming neighbors.  Thus, the in-edges of $v_i$ in the dependency graph $D$ are unnecessary for the control model~\eqref{eq:controlmodel}. We thus modify the definition of the dependency graph to accommodate the existence of control nodes by deleting the in-edges of each control node in $V^*$. Specifically, we have the following definition:

\begin{definition}[Derived graph~\cite{statecontrol}]\label{subgraph}
	Let $D=(V,E)$ be the dependency graph associated with a CBN. Let $V^*\subset V$ be the set of control nodes associated with system~\eqref{eq:controlmodel}.  The {\bf derived graph} $D' = (V, E')$ is a digraph, with $V$ the node set and $ E' =E \setminus \cup_{u\in V^*}\Ein(u)$ the edge set.
\end{definition}


\section{Orbit-controllability}\label{orbit}
We investigate in this section the orbit-controllability of a CBN.  
{\color{black} To proceed, we first note that the asymptotic behavior of a CBN was investigated mostly over strongly connected digraphs, and little is known for other cases. In particular, it is known that the periodic orbits of strongly connected CBNs can be identified with binary necklaces of a certain length: Let $D = (V, E)$ be strongly connected, and denote by $D_1 = (V_1, E_1),\ldots, D_N = (V_N, E_N)$, with $V_i \subset V$ and $E_i\subset E$, the cycles of $D$. Let $n_i$ be the length of $D_i$, and~$p^*$ be the greatest common divisor of $n_i$, for $i = 1,\ldots, N$: $$p^*:= {\rm gcd}\{n_1, n_2, ..., n_N\},$$ 
which is also known as the \emph{loop number} of $D$~\cite{colon2005boolean}.  We need the following fact:

\begin{lemma}\label{bijec}
If the dependency graph is strongly connected, then the period of the associated CBN is a divisor of~$p^*$. 
Furthermore, there is a bijection between the set of periodic orbits and the set of binary necklaces of length $p^*$: We identify a periodic orbit $\{x(t_0), \ldots, x(t_0 + p  - 1)\}$ with the  corresponding binary necklace $x_i(t_0)x_i(t_0 + 1)\ldots x_i(t_0 + p^* - 1)$, where the choice of a vertex $v_i$ can be arbitrary. 
\end{lemma}
}

We refer the reader to~\cite{jarrah2010dynamics,gao2016periodic,stabilityfull} for proofs of Lemma~\ref{bijec}. For the remainder of the paper, we let $S$ be the set of periodic orbits. Note, in particular, that from Lemma~\ref{bijec} the two binary necklaces $s=0\ldots 0$ and $s=1\dots 1$ correspond to the fixed points $x = (0,\ldots, 0)$ and $x = (1,\ldots, 1)$, respectively. We further introduce the following definition: 
With the preliminaries above, we establish the first main result of the paper:

{\color{black}
\begin{theorem}\label{thm1}
	Let the dependency graph $D=(V,E)$ of a conjunctive Boolean network be strongly connected. Then, a subset $V^*$ is an {\color{black}orbit-controlling set} if and only if the associated derived graph $D'$ is acyclic.  
\end{theorem}

\begin{Remark}
	Recall that a source node is defined as a vertex with no in-edges. Since $D$ is strongly connected, there is no source node in $D$. In $D'$, however, we have eliminated all in-edges of vertices in $V^*$. Thus, if $D'$ is acyclic, then the nodes in $V^*$ are necessarily the source nodes of $D'$ and vice versa. 
\end{Remark} 

Recall that $V_1,\ldots,V_N$ are the vertex sets of the cycles of $D$. Then, the statement of Theorem~\ref{thm1} is equivalent to the following statement:  $V^*\subseteq V$ is an orbit-controlling set if and only if
	\begin{eqnarray}\label{empty}
		V^*\cap V_i\neq \varnothing,\ \ \forall i=1,\ldots,N
	\end{eqnarray}
	}

{\color{black}
{\em Illustration of Theorem~\ref{thm1}.} We consider here a CBN with two different sets of control nodes, as shown in Fig.~\ref{Vbar}.
The associated derived graphs are shown in Fig.~\ref{DAG}, which are acyclic. Thus in both cases, the control nodes (vertices colored blue) form an orbit-controlling set. To check~\eqref{empty}, we note that there are two cycles in the graph, whose vertex sets are $V_1=\{v_1,v_2,v_3,v_4,v_5,v_6\}$ and $V_2=\{v_1,v_2,v_7,v_8\}$, respectively. On the left of Fig.~\ref{Vbar} (and Fig.~\ref{DAG}), $V^*=\{v_2\}$, and thus $$V^*\cap V_1=V^*\cap V_2=\{v_2\}\neq \varnothing.$$ On the right of Fig.~\ref{Vbar} (and Fig.~\ref{DAG}), $V^*=\{v_4,v_7\}$, and thus $$V^*\cap V_1=\{v_4\}\neq \varnothing,\hspace{0.5 cm} V^*\cap V_2=\{v_7\}\neq \varnothing.$$
}

\begin{figure}[h]
	\centering
	\includegraphics[height=50mm]{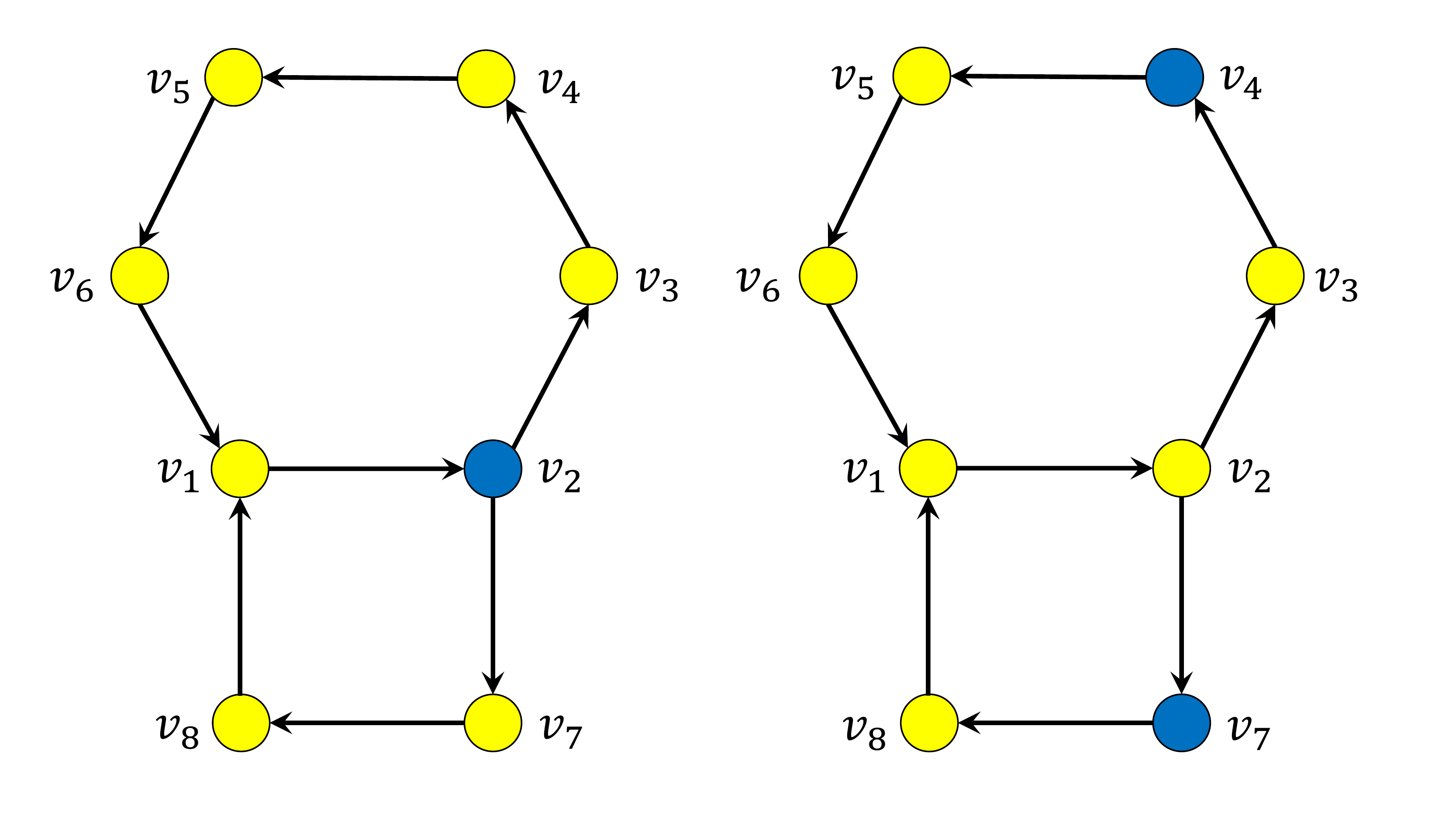}
	\caption{The above are two examples of orbit-controlling sets. Vertices colored blue are in the orbit-controlling set. The graph has two cycles. In the left figure, the only vertex in the orbit-controlling set is shared by both cycles. In the right figure, we have picked one vertex in each cycle to be in the orbit-controlling set.}
	\label{Vbar}
\end{figure}


\begin{figure}[h]
	\centering
	\includegraphics[height=50mm]{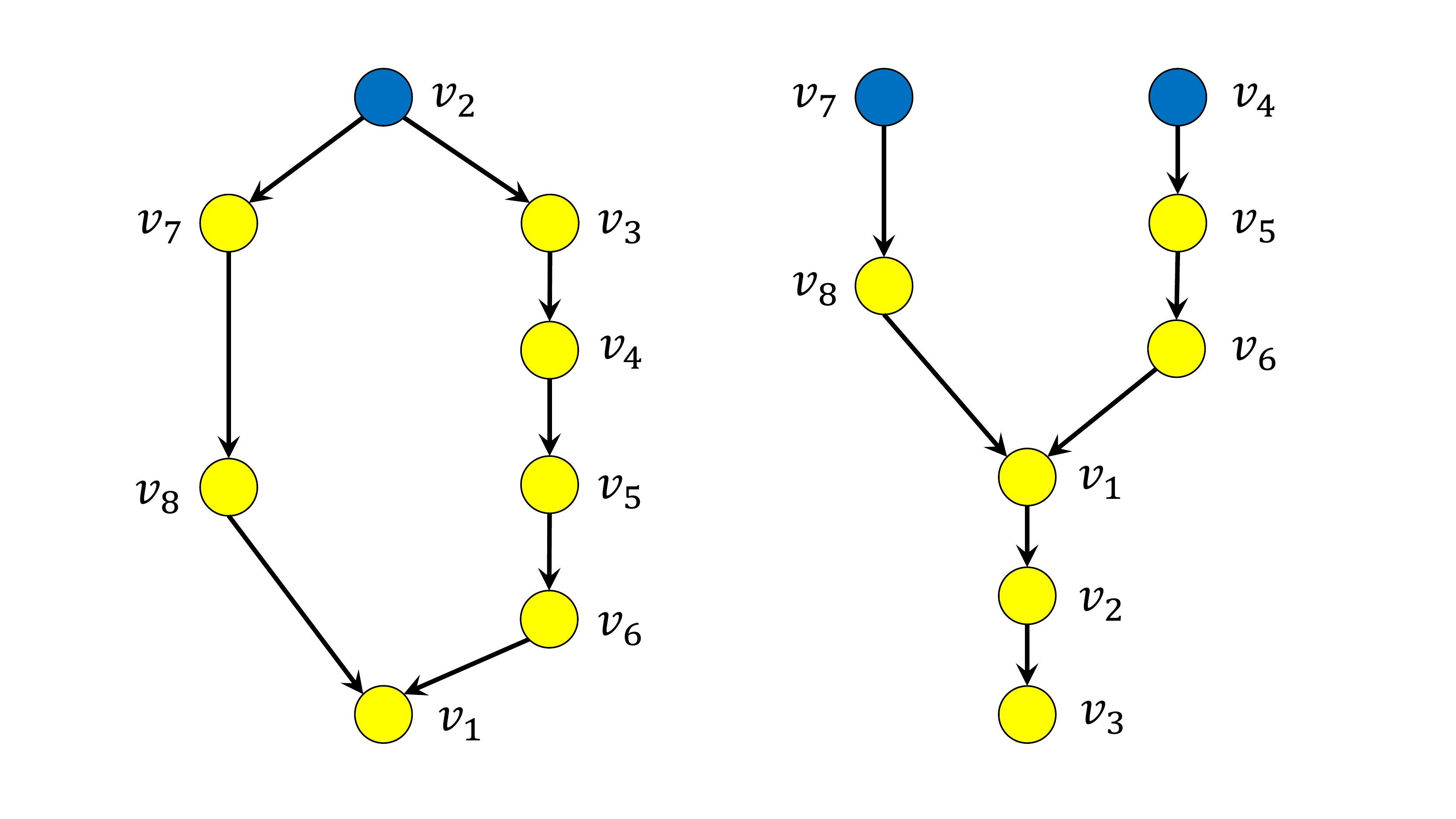}
	\caption{The above are two examples of $D'$. The left (right) figure is obtained by removing the in-edges of vertices in the orbit-controlling set in the left (right) figure of Fig.~\ref{Vbar}. It can be seen that the $D'$ obtained this way is acyclic, and the set of source nodes is exactly the orbit-controlling set.}
	\label{DAG}
\end{figure}

The remainder of this section is devoted to the proof of Theorem~\ref{thm1}. {\color{black} We first introduce a notation: For a subset $V' = \{v_{i_1},\ldots, v_{i_m}\}$ of $V$, we define $x_{V'} := (x_{i_1},\ldots, x_{i_m})$. 
}
We then first prove the necessity, i.e., if $V^*$ is an orbit-controlling set, then $V^*\cap V_i\neq \varnothing$, $\forall i=1,\ldots,N$. 

\begin{proof}[Proof of necessity of~\eqref{empty}]
	The proof is carried out by contradiction. Suppose to the contrary that for some cycle $D_i$, $V^*\cap V_i=\varnothing$. Then, given an initial condition $x(0)=(0,\ldots,0)\in \mathbb{F}^n_2$, it is never possible for the trajectory to reach the periodic orbit $s=1\ldots 1$. To see this, recall that $s=1\ldots 1$ corresponds to the fixed point $x=(1,\ldots,1)$, which is the only state in $s$. Then, for each vertex $v_j\in V_i$, there is a vertex $v_k\in V_i$ such that {\color{black}$v_k\in \Nin(v_j)$}. Since $x_k(0)=0$, $x_j(1)=0$ by the value update rule. Thus, {\color{black}$$x_{V_i}(t)=x_{V_i}(t-1)=\cdots=x_{V_i}(0)=(0,\ldots,0),$$}which implies that the trajectory will never enter $s=(1,\ldots,1)$. This contradicts our initial assumption that $V^*$ is an orbit-controlling set.
\end{proof}

We next prove the sufficiency, i.e., if~\eqref{empty} is satisfied, then $V^*$ is an orbit-controlling set. We will first provide an algorithm, Algorithm~\ref{orbitalgo}, in which we assign values to the control nodes (i.e., the entries of $x_{V^*}$) along time so that the trajectory generated by the control system, with any given initial condition $x(0)$, will enter the desired periodic orbit $s=y_0\ldots y_{p^*-1}$. The algorithm is comprised of two parts. The first part is from line~2 to line~7, where we always assign ``$1$" to all entries of $x_{V^*}$ until the trajectory enters the periodic orbit $s'=1\ldots1$. 
{\color{black} We note that from a biological perspective, assigning ``$1$" to a vertex~$v_i$ means  providing the product of the corresponding gene~$i$ (usually proteins) to the system. 
Equivalently, the gene~$i$ can be equivalently viewed as at ``on" status in the system.}
The second part is from line~8 to line~11, where we sequentially assign the values from the desired periodic orbit (represented by a binary necklace $y_0\ldots y_{p^*-1}$) to any single vertex in $V^*$. 

\begin{algorithm}
	\caption{Control law for orbit-controlling}\label{orbitalgo}
	\begin{algorithmic}[1]
		\Procedure{Control}{$V^*,s$} 
		\State $t\gets 0$ 
		\While{$x(t)\not=(1,\ldots,1)\in \mathbb{F}^n_2$} 
		\State $x_{V^*}(t)\gets (1,\ldots,1)$
		\State $t\gets t+1$
		\EndWhile\label{euclidendwhile} 
		\State \textbf{end while}
		\State $\tau\gets t$
		\State \textbf{pick any $v_i\in V^*$}
		\For{$t' := 0 \ \textbf{to}\  p^*-1$}
		\State $x_i(\tau+t')\gets y_{p^*-1-t'}$;
		\EndFor
		\State \textbf{end for}	
		\EndProcedure
		\State \textbf{end procedure}
	\end{algorithmic}
\end{algorithm}

{\color{black}
	
	{\em Illustration of Algorithm~\ref{orbitalgo}.} We consider the CBN whose dependency graph is shown in Fig.~\ref{Vbar}. 
	The loop number $p^*$ is~$2$, and hence a periodic orbit of the system is identified with a binary necklace of length~$2$. 
	Suppose that the desired periodic orbit is $s=01$. Then, for the control system on the left of Fig.~\ref{Vbar} with $V^* = \{v_2\}$, the control inputs obtained from Algorithm~\ref{orbitalgo}	are given by
	\begin{center}
		\begin{tabular}{l*{7}{c}r}
			Step $t$          & 0 & 1 & 2 & 3 & 4 & 5 & 6 & 7\\
			\hline
			$x_2(t)$		  & 1 & 1 & 1 & 1 & 1 & 1 & 0 & 1\\
		\end{tabular}
	\end{center}
{\color{black} In this case, $\tau=6$. The system will enter the periodic orbit $s=01$ at time step $(\tau+7)$ as illustrated in Fig.~\ref{3B}.

\begin{figure}[h]
	$\hspace{-0.04cm}\includegraphics[height=49mm]{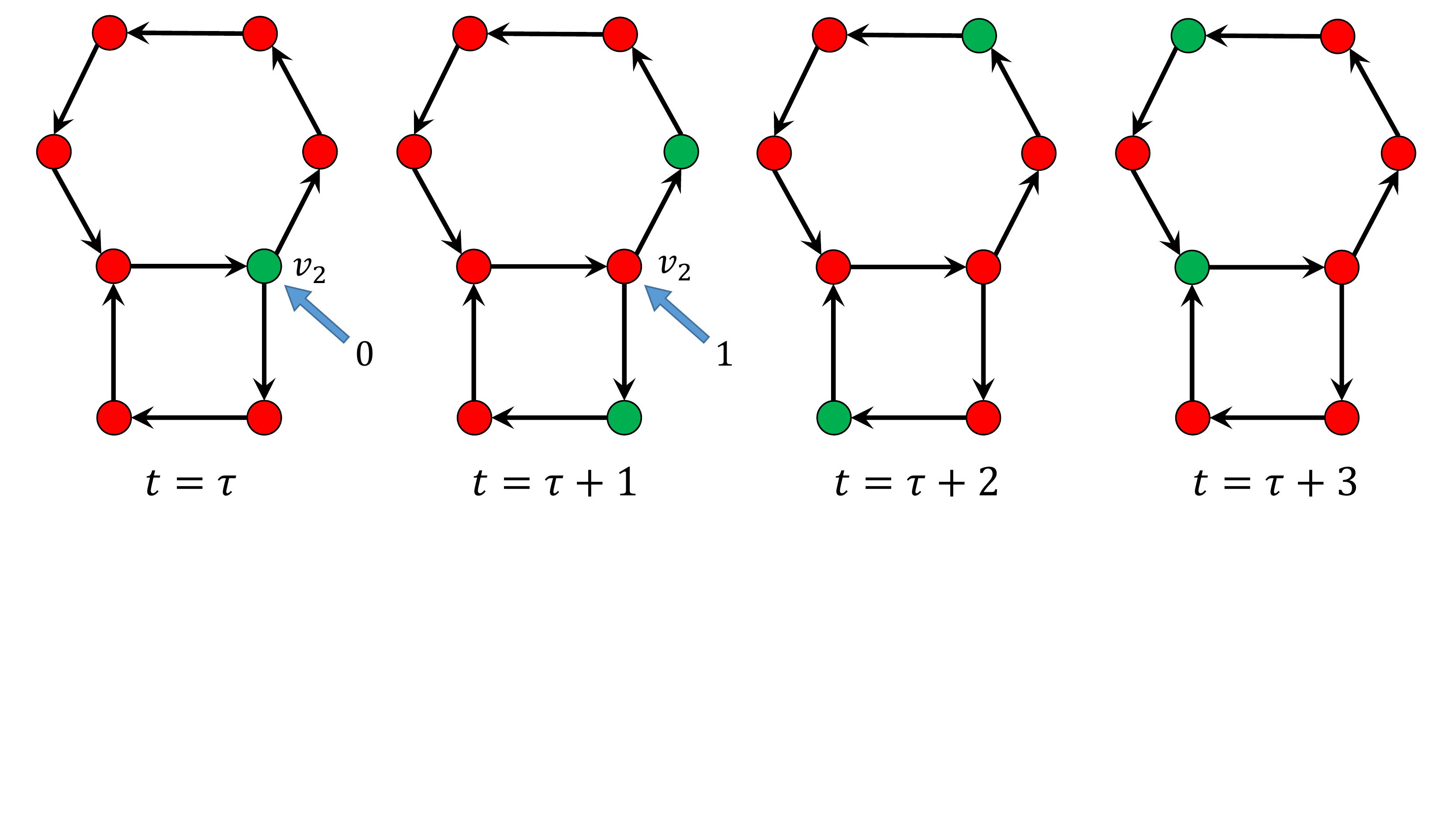}\vspace{-0.2 cm}$
	$\vspace{-1.5 cm}$	
	$\includegraphics[height=49mm]{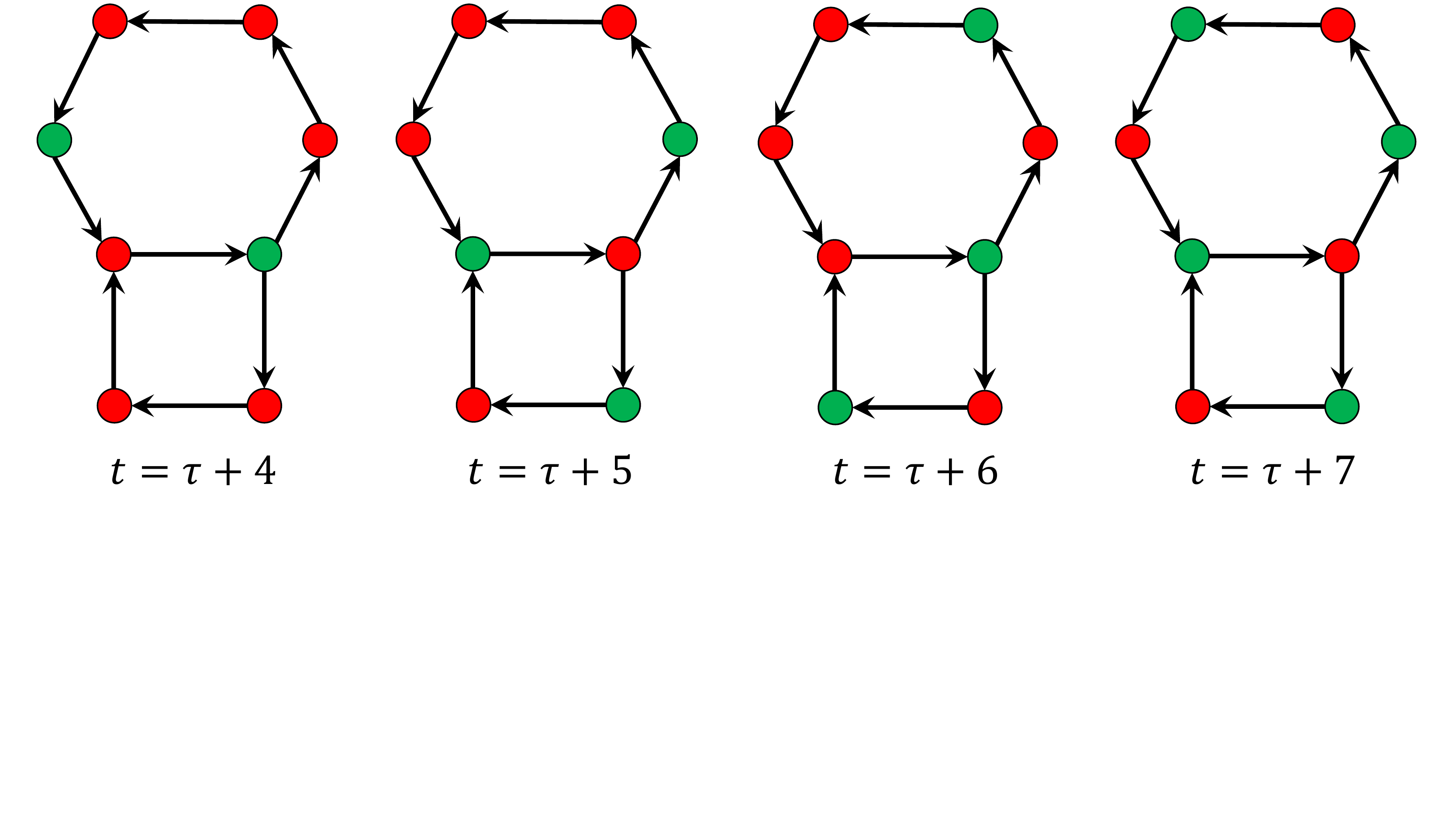}\vspace{-2 cm}$	
	\caption{{\color{black} The above figure illustrates the second part of the control procedure described in Algorithm~\ref{orbitalgo}. Specifically, it shows the system states from $t=\tau$ to $t=\tau+7$. We use the red (resp. green) color to denote that the corresponding node is holding value ``1" (resp. ``0"). We assign to the node $v_2$ the values $0$ and $1$ at the time steps $t= \tau$ and $t = \tau+1$, respectively. With these assignments,  the system will enter the periodic orbit $s=01$ at the time step $t=\tau+7=13$.}}
	\label{3B}
\end{figure}
}

	For the control system on the right of Fig.~\ref{Vbar} with $V^* = \{v_4, v_7\}$, the control inputs are given by
	\begin{center}
		\begin{tabular}{l*{7}{c}r}
			Step $t$          & 0 & 1 & 2 & 3 & 4 & 5 & 6 & 7\\
			\hline
			$x_4(t)$	      & 1 & 1 & 1 & 1 & 1 & 1 & 1 & 1\\
			$x_7(t)$          & 1 & 1 & 1 & 1 & 1 & 1 & 0 & 1\\
		\end{tabular}
	\end{center}
	In either case, the control inputs will drive the system from any initial condition to enter the periodic orbit~$s$. 
}

{\em Validating Algorithm~\ref{orbitalgo}.} 
According to Algorithm~1, the proof of the validity is divided into two parts.   

\paragraph{Part~I: Driving the system to the state $x=(1,\ldots,1)$}\label{part1Alg1}
We show here that the first part of Algorithm~\ref{orbitalgo} (specifically, the ``while" loop) will be terminated in at most~$n$ time steps: 
{\color{black}
\begin{proposition}\label{while}
	If the derived graph $D'$ associated with the control system~\eqref{eq:controlmodel} is acyclic, then by setting $u_i(t) = 1$ for all $v_i \in V^*$ and $0\le t\le n-1$, we have that $x(n-1) = (1,\ldots,1)$. In particular, $\tau \le (n-1)$.
\end{proposition}
}

\begin{proof}
	Suppose that, to the contrary, $x(n-1)\neq (1,\ldots,1)$. Without loss of generality, take $x_i(n-1)=0$. Since the value of each control node is fixed to be ``$1$", $v_i\notin V^*$, and hence $\Nin(v_i; D')\neq \varnothing$. By value update rule, there exists a vertex $v_{i_1}\in\Nin(v_i; D')$ with $x_{i_1}(n-2)=0$. Similarly, we have that $v_1\notin V^*$ and there exists a vertex $v_{i_2}\in \Nin(v_1;D')$ with $x_{i_2}(n-3)=0$. Repeating this argument, we find vertices $v_{i_1},\ldots,v_{i_{n-1}} \notin V^*$ such that 
	{\color{black}
	$$
	x_i(n-1)=x_{i_1}(n-2)=\cdots=x_{i_{n-1}}(0)=0,
	$$}On the other hand, there are only $n$ vertices in $D'$. We thus have $v_{i_j}=v_i$ for some $j\in \{1,\ldots,n-1\}$. But then, there is a cycle $v_{i_j}v_{i_{j-1}}\ldots v_{i_1}v_i$ in $D'$ which is a contradiction.  
\end{proof}


\paragraph{Part II: Driving the system from $x=(1,\ldots,1)$ to the periodic orbit $s$}

We show here that after performing the ``for" loop of Algorithm~\ref{orbitalgo}, the trajectory of the system states will enter the periodic orbit $s$.
Recall that $s$ is represented by a binary necklace of length $p^*$: $s=y_0\ldots y_{p^*-1}$. If $s=1\ldots 1$, then we are done by the first part of the Algorithm~1 (lines~2-7). Otherwise, we need to execute the second part of the algorithm (lines~8-11). 
As a result, we provide the following proposition, whose proof is given in the Appendix.

\begin{proposition}\label{initial}
	Fix a vertex $v_i\in V$, and write $s=y_0\ldots y_{p^*-1}$. After executing the control law given in Algorithm~\ref{orbitalgo}, the state $x$ at time $\tau+p^*-1$ is given by
	\begin{equation}\label{beginfor}
	\begin{array}{r@{}l}
	x_{\N^{j}(v_i)}(\tau+p^*-1) &\ =y_{j}{\bf 1} ,\ \forall j = 0,\ldots, p^*-1\\
	x_r(\tau+p^*-1) &\ =1, \hspace{7pt} \forall v_r\notin \cup_{j=0}^{p^*-1}\N^{j}(v_i)
	\end{array}		
	\end{equation}
	where {\bf 1} is a vector of all ones with an appropriate dimension. Moreover, a trajectory generated by the system~\eqref{eq:updaterule}, with the initial condition~\eqref{beginfor}, will enter the periodic orbit $s$ after finite time steps.
\end{proposition}


%

\begin{Remark}
	Recall that the ``while" loop takes a maximum of $(n-1)$ time steps, and the ``for" loop takes $p^*$ time steps. Therefore, the maximum total time it takes to control the network is $(n+p^*-1)$. The time it takes for the system to finally enter the periodic orbit, however, can be longer.
\end{Remark}

Combining Proposition~\ref{while} and Proposition~\ref{initial} leads to the sufficiency part of Theorem~\ref{thm1}.


\section{State-controllability}\label{state}

In this section, we investigate the state-controllability of a CBN. {\color{black} We do {\em not} require that the dependency graph $D$ be strongly connected.} The main result of the section is stated as follows:

\begin{theorem}\label{thm2}
	Let $D=(V,E)$ be the dependency graph associated with a conjunctive Boolean network. 
	A subset $V^*\subseteq V$ is a state-controlling set if and only if the associated derived graph $D'$ satisfies the following conditions:
	\begin{enumerate}
		\item The derived subgraph $D'$ is acyclic. \label{item1}
		\item For any $v\in V$, there exists a control node $u\in V^*$ and an integer~$k\ge 0$ such that $\N^k(u; D')=\{v\}$. \label{item2}
	\end{enumerate}\,
\end{theorem}

Note that  the first item of Theorem~\ref{thm2} is itself a necessary and sufficient condition for $V^*$ to be an orbit-controlling set. 
The second item is thus a necessary and sufficient condition for an orbit-controlling set to be a state-controlling set.

{\color{black}
{\em Illustration of Theorem~\ref{thm2}.} We consider again the example shown in Fig.~\ref{Vbar}, where we have a CBN with two different sets of control nodes. Recall that the associated derived graphs are acyclic in both cases (given in Fig.~\ref{DAG}). Thus, the two sets of control nodes are both orbit-controlling sets.} However, only the control nodes on the right of Fig.~\ref{Vbar} form a state-controlling set. Indeed, we have 
$$
\begin{array}{l}
\vspace{0.1 cm}
\N^2(v_7) = \N^3(v_4) =\{v_1\}, \\
\vspace{0.1 cm}
\N^3(v_7) = \N^4(v_4) =\{v_2\},  \\
\vspace{0.1 cm}
\N^4(v_7) = \N^5(v_4) =\{v_3\}, \\
\vspace{0.1 cm}
\N^0(v_4) = \{v_4\},  \\
\vspace{0.1 cm}
\N^1(v_4) = \{v_5\},  \\
\vspace{0.1 cm}
\N^2(v_4) = \{v_6\}, \\
\vspace{0.1 cm}
\N^0(v_7) = \{v_7\}, \\
\vspace{0.1 cm}
\N^1(v_7) =  \{v_8\}, 
\end{array}
$$  
where all the out-neighbors are taken within $D'$. Thus, the second condition of Theorem~\ref{thm2} is satisfied. {\color{black} On the other hand, 
the set of control nodes on the left of {\color{black}Fig.~\ref{Vbar}} is {\em not} a state controlling set. To see this, we note that the node $v_4$ of the left DAG only lies in $\N^2(v_2)$, but $\N^2(v_2)=\{v_4,v_8\}\neq\{v_4\}$}, and hence the second condition of Theorem~\ref{thm2} is not satisfied.


We prove in the remainder of this section Theorem~\ref{thm2}. 
The necessity and sufficiency of the two conditions listed in Theorem~\ref{thm2} are established subsequently in the following subsections.

\subsection{Necessity}
We prove here the necessity part of Theorem~\ref{thm2}. Specifically, we show that if $V^*$ is a state-controlling set, then the two conditions in Theorem~\ref{thm2} must hold. The necessity of the first condition should be clear as a state-controlling set is necessarily an orbit-controlling set. 


We establish below the necessity of the second condition. The proof will be carried out by contradiction. Specifically, we assume that the derived graph $D'$ is a DAG which does not satisfy the second item in Theorem~\ref{thm2}. We then show that system~\eqref{eq:controlmodel} is not controllable. 
To proceed, we first have some preliminaries on the control dynamics~\eqref{eq:controlmodel}. From~\eqref{xtneighbor}, we have that for any $v_i\notin V^*$,
{\color{black}
$$
x_i(t) =  \prod_{v_j\in \mathcal{N}_{\rm in}(v_i; D')}x_j(t-1).
$$
For each $v_j\in \Nin(v_i;D')$, we have two cases: If $v_j$ is a control node, then we keep the factor $x_j(t-1)$ in~\eqref{xtneighbor}. If $v_j$ is not a control node, then $v_j$ has a nonempty set of incoming neighbors. We can thus appeal again to~\eqref{xtneighbor} and replace the factor $x_j(t-1)$ in~\eqref{xtneighbor} with the following expression: 
$$
x_j(t-1) = \prod_{v_k\in \mathcal{N}_{\rm in}(v_j;D')}x_k(t-2).$$}Since $D'$ is a DAG, by recursively applying the arguments above, we obtain that 
\begin{eqnarray}\label{sourceprod}
x_i(t)= \prod_{v_j\in V^*_i} \prod_{p \in P_{ji}} x_j(t-l(p)),
\end{eqnarray}
where $V^*_i\subseteq V^*$ is a subset of the set of source nodes such that there is at least one path from $v_j$ to $v_i$ for all $v_j\in V^*_i$. We recall that $P_{ji}$ is the set of paths (within $D'$) from $v_j$ to $v_i$ and $l(p)$ is the length of path $p$. Since the nodes $v_j$'s in~\eqref{sourceprod} are the control nodes of $D'$, we call~\eqref{sourceprod} the {\bf control expression} of $x_i(t)$. In Fig.~\ref{expression}, we provide an example where we write the values of all nodes in their control expression form.
\begin{figure}[h]
	\centering
	\includegraphics[height=50mm]{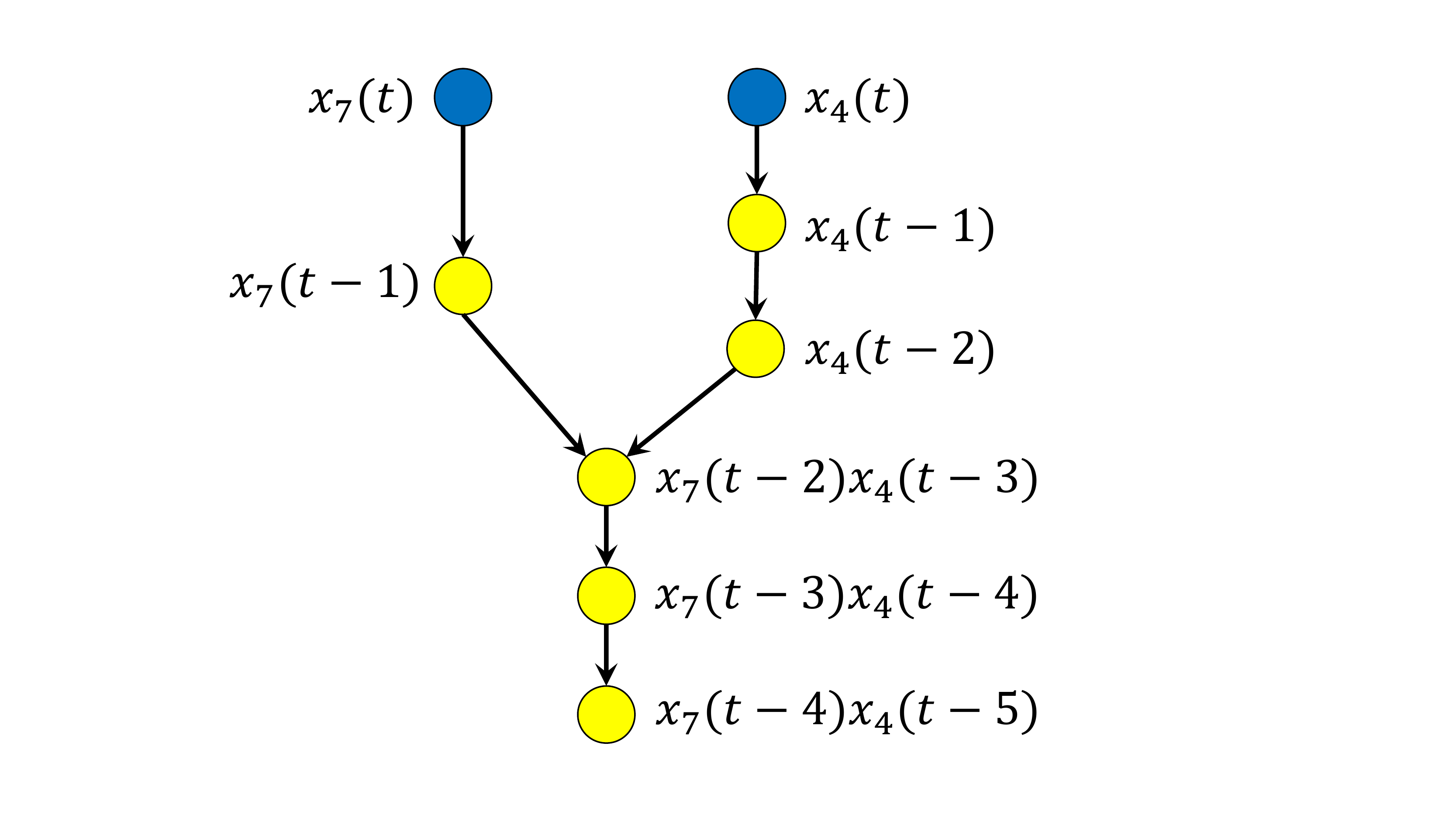}
	\caption{The DAG in this figure is the derived graph of the dependency graph shown on the right of Fig.~\ref{Vbar}. The two nodes $v_7$ and $v_4$ (marked in blue) form a state-controlling set. The values of all nodes at time $t$ are expressed in their control expression form.}
	\label{expression}
\end{figure}

With the preliminaries above, we are now in a position to prove the necessity of the second condition of Theorem~\ref{thm2}.

\begin{proof}[Proof of necessity of condition~\ref{item2}]
	Let $v_i\in V$ be a node such that $\N^k(u;D')\neq \{v_i\}$ for any $u\in V^*$ and any $k\geq 0$. We now show that system~\eqref{eq:controlmodel} cannot be driven from an initial state $(1,\ldots,1)$ 
	to the final state $x^*$ where $x^*_i = 0$ and $x^*_s= 1$ for all $v_s \neq v_i$. The proof is carried out by contradiction, i.e., we assume that there is a set of  control laws using which we can steer the system to reach $x(t) = x^*$ for some~$t \ge 0$.

	We let the control expression of $x^*_i(t)$ be given by~\eqref{sourceprod}. 
	We then pick an arbitrary factor in~\eqref{sourceprod}, say $x_j(t-l(p_1))$, with $v_j\in V^*_i\cap \Nin^{l(p_1)}(v_i;D')$. By assumption, we have $\N^{l(p_1)}(v_j;D')\neq \{v_i\}$. Thus, there exists a node $v_s$, other than $v_i$, such that $v_s\in \N^{l(p_1)}(v_j;D')$. We then apply the control expression to $x^*_s$. 
	Note, in particular, that the factor $x_j(t-l(p_1))$ we picked in the control expression of $x^*_i(t)$ is also a factor in the control expression of $x^*_s(t)$.  Moreover, since $v_s \neq v_i$ and $x^*_s(t)=1$, it is necessary that  $x_j(t-l(p_1))=1$. Since the factor $x_j(t-l(p_1))$ in the control expression of $x^*_i(t)$ is picked arbitrarily, it is necessary that any such factor holds value ``$1$". Thus, $x^*_i(t) = 1$, which is a contradiction. This completes the proof. 
\end{proof}

\subsection{Sufficiency}

We next prove the sufficiency part of Theorem~\ref{thm2}. Specifically, we show that if $V^*\subseteq V$ satisfies the two conditions listed in Theorem~\ref{thm2}, then $V^*$ is a state-controlling set. The proof will be carried out by exhibiting an explicit control law for steering the system from an arbitrary initial condition  to  the desired final state $x^*$. To the end, let $T$ be the length of a longest path in the derived graph $D'$. The following algorithm assigns the values to $x_{V^*}(t)$, for $0\le t \le T$, such that the trajectory generated by the control system~\eqref{eq:controlmodel}, from an arbitrary initial condition, reaches $x^*$ at time $T$. 

\begin{algorithm}[h]
	\caption{Control law for state-controlling}\label{statealgo}
	\begin{algorithmic}[1]
		\Procedure{Control}{$V^*,x^*$} 
		\State $T\gets$ length of the longest path in $D'$
		\For{$t := 0 \ \textbf{to}\  T$}
		\For{$v_i\in V^*$}
		\If{$|\N^{T-t}(v_i; D')|==1$}
		\If	{$x^*_{\N^{T-t}(v_i;D')}==0$}
		\vspace{0.1cm}
		\State $u_i(t)\gets 0$
		\State \textbf{continue}
		\EndIf
		\State \textbf{end if}
		\EndIf
		\State \textbf{end if}
		\State $u_i(t)\gets 1$
		\EndFor
		\State \textbf{end for}
		\EndFor
		\State \textbf{end for}
		\EndProcedure
		\State \textbf{end procedure}
	\end{algorithmic}
\end{algorithm}

The assignment of Algorithm~\ref{statealgo} can be interpreted as follows: At time step~$t$ and for each control node $v_i\in V^*$, there are two cases:  If there exists a node $v_j\in V$ such that $\N^{T-t}(v_i)=\{v_j\}$ and $x^*_j=0$, then we let $u_i(t)=0$. Otherwise, we let $u_i(t)=1$. We also note that the values of control nodes assigned by the algorithm above do {\em not} depend on the initial condition. 
%

{\em Illustration of Algorithm~\ref{statealgo}.} We consider the CBN whose dependency graph (resp. derived graph) is shown on the right of Fig.~\ref{Vbar} (resp. Fig.~\ref{DAG}). 
Suppose that the desired final state is $x^*=\{x^*_1,\ldots,x^*_8\}=\{1,1,0,0,0,1,0,1\}$; then, the control inputs for $x_4$ and $x_7$ obtained from Algorithm~\ref{statealgo} are given by:
	\begin{center}
		\begin{tabular}{l*{6}{c}r}
			Step $t$          & 0 & 1 & 2 & 3 & 4 & 5  \\
			\hline
			$x_4(t)$			  & 0 & 1 & 1 & 1 & 0 & 0   \\
			$x_7(t)$             & 1 & 0 & 1 & 1 & 1 & 0   \\
		\end{tabular}
	\end{center}\,
{\color{black} With these inputs, the system will enter the state $x^*=\{x^*_1,\ldots,x^*_8\}=\{1,1,0,0,0,1,0,1\}$ at time step $t=5$ as illustrated in Fig.~\ref{4B}.
	
	\begin{figure}[h]
		$\hspace{-0.04cm}\includegraphics[height=49mm]{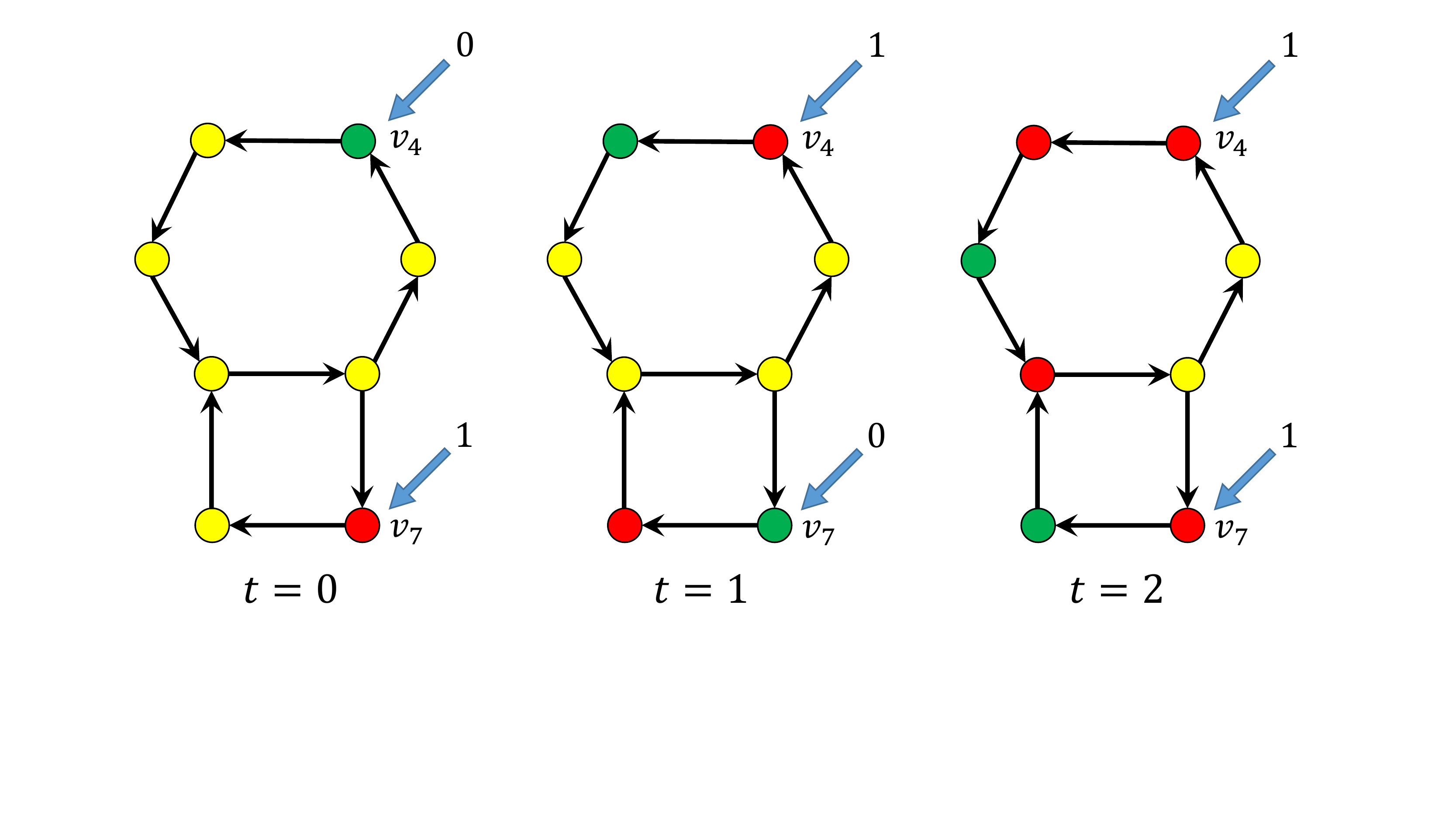}$
		$\vspace{-1.2 cm}$	
		$\includegraphics[height=49mm]{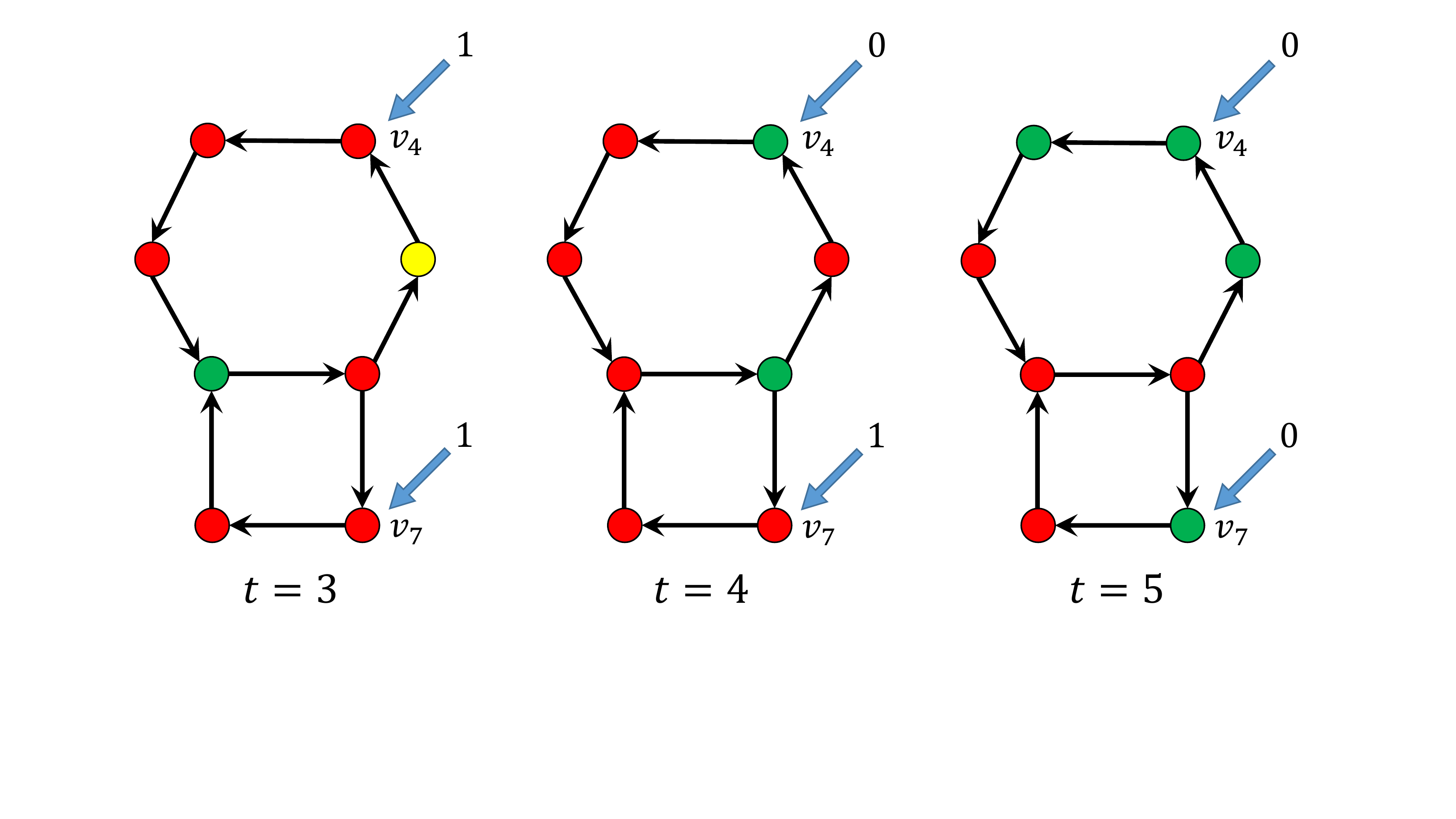}\vspace{-1.2 cm}$	
		\caption{{\color{black} The above figure illustrates the control procedure described in Algorithm~\ref{statealgo}. Specifically, it shows the system states from $t=0$ to $t=5$. We use the red (resp. green) color to denote that the corresponding node is holding value ``1" (resp. ``0"). Vertices are colored yellow if their values are irrelevant, i.e., their values do not affect the control procedure. We assign to the nodes $v_4$ and $v_7$ at the time steps $t= 0$ to $t = 5$. With these assignments, the system will enter the state $x^*=\{1,1,0,0,0,1,0,1\}$ at the time step $t=5$.}}
		\label{4B}
	\end{figure}
}

%

{\em Validating Algorithm~\ref{statealgo}.} 
We show below that for any $v_j\in V$, the Algorithm~\ref{statealgo} leads to $x_j(T)=x_j^*$. There are two cases.
	
	\emph{Case I:} $x_j^*=0$.
	If $v_j\in V^*$, then $\N^{T-T}(v_j)=\{v_j\}$, and both ``if" conditions in Algorithm~\ref{statealgo} are satisfied. Thus, we have that $$x_j(T)=u_j(T)=0.$$  
	If $v_j\notin V^*$, then by the second condition in Theorem~\ref{thm2}, there exists a control node $v_i\in V^*$ and an integer $k$, with $0<k\le T$, such that $\N^k(v_i;D')=\{v_j\}$. At time $t=T-k$, we have that $|\N^{T-t}(v_i)|=1$. Both ``if" conditions in Algorithm~\ref{statealgo} are satisfied. Thus, $$x_i(T-k)=u_i(T-k)=0.$$ Also, $\N^k(v_i;D')=\{v_j\}$ indicates that there is a path (within $D'$) of length $k$ from $v_i$ to $v_j$. Appealing to~\eqref{sourceprod}, we obtain that $x_i(T-k)$ is a factor of the control expression of $x_j(T)$, which leads to $$x_j(T)=x_i(T-k)=0.$$
	
	\emph{Case II:} $x_j^*=1$. 
	From the control expression~\eqref{sourceprod}, we obtain
	$$
	x_j(T)= \prod_{v_i\in V^*_j} \prod_{p\in P_{ij}} x_i(T-l(p)).
	$$
	Note that $l(p)\le T$ because $T$ is the length of a longest path in $D'$. 
	It now suffices to show that each factor $x_i(T-l(p))$ above is assigned the value ``$1$" under Algorithm~\ref{statealgo}. Note that there is a path of length $l(p)$ from $v_i$ to $v_j$, i.e., $v_j\in\N^{l(p)}(v_i;D')$. If $|\N^{T-(T-l(p))}(v_i)|\neq 1$, then the ``if" condition in line~5 of Algorithm~\ref{statealgo} is not satisfied. Thus, by the value assignment rule in line~11, we have that $$x_i(T-l(p))=u_i(T-l(p))=1.$$ 
	If $\N^{T-(T-l(p))}(v_i;D')= \{v_j\}$, then the ``if" condition in line~5 of Algorithm~\ref{statealgo} is satisfied. However, since $x_j^*=1$, the ``if" condition in line~6 is not satisfied. Thus, by the value assignment rule in line~11, we again have that $$x_i(T-l(p))=u_i(T-l(p))=1.$$
This then establishes the validity of Algorithm~\ref{statealgo}. We thus complete the proof of Theorem~\ref{thm2}. \hfill{\qed}



\section{Conclusions and Outlooks} \label{end}

In this paper, we have posed and answered the following two-part controllability question: Given a subset of nodes of the dependency graph, what are the necessary and sufficient conditions for a subset to be an orbit-controlling set or a state-controlling set? The answers were given in Theorem~\ref{thm1} and Theorem~\ref{thm2}. In particular, we related the orbit-controllability as well as controllability of system~\eqref{eq:controlmodel} to the structure of the derived graph. We have also presented, in Algorithm~\ref{orbitalgo} (resp. Algorithm~\ref{statealgo}), a method of assigning the values of the control inputs to steer system~\eqref{eq:controlmodel} to a desired periodic orbit (resp. final state). Algorithm~\ref{orbitalgo} takes at most $(n+p^*-1)$ time steps, with $n$ being the number of vertices in the dependency graph and $p^*$ the greatest common divisor of cycle lengths. Algorithm~\ref{statealgo} takes at most $T$ time steps, with $T$ being the length of the longest path in the derived graph.   

Although systems biology serves as the main motivation for our research,  applications of this work are by far not limited to gene regulated networks. CBNs are also suitable to model, for example, water quality networks. In such a network, each Boolean variable can be viewed as the water quality within a pipe. The Boolean variable takes the value ``1" if the water is not polluted, and the value ``0" if the water is polluted. The water in each pipe comes from some other pipes, and is polluted if the water in one of those other pipes was polluted. Other examples which can be modeled by CBNs include social networks (information flow on Twitter or Facebook), and supply chain networks (movement of materials), and the results of this paper would also apply to all these networks.

{\color{black}
There are several research directions we will pursue in our future work. First, recall that in the study of orbit-controlling sets, we considered only CBNs whose dependency graphs are strongly connected, because the periodic orbits of CBNs with weakly connected dependency graphs have not yet been fully characterized. Most recently, we have made some progress in this direction in~\cite{chen2017asymptotic}, where we have investigated the asymptotic behavior of weakly connected CBNs. We plan to generalize the result of orbit-controllability obtained in this paper to a general weakly connected dependency graph. 

{\color{black}Second, we plan to develop algorithms for 1) finding all orbit- and state-controlling sets of a CBN; 2) finding an orbit-controlling set and/or a state-controlling set with minimal cardinality.}
We note that finding the orbit-controlling set with minimum cardinality is in fact equivalent to finding the minimum cardinality of the so-called {\bf feedback vertex set}, the set of vertices (nodes) whose removal leads to DAG. This problem has been shown to be NP-hard for general graphs in~\cite{karp1972reducibility}, and it has been shown in~\cite{fomin_et_al:LIPIcs:2010:2470} that finding a minimum feedback vertex set of general undirected graphs with $n$ nodes can be solved in time $\mathcal{O}(1.7347^n)$. For general directed graphs, an algorithm has been provided in~\cite{razgon2007computing}, solving the problem in time $\mathcal{O}(1.9977^n)$. A faster algorithm for finding the minimum feedback vertex set in strongly connected graphs may be developed in the future. An algorithm for finding the minimum state-controlling set may be developed as well. 

Third, we plan to explore the tradeoff between the number of control nodes and the time it takes for the system to reach a desired state (or an periodic orbit). Controllability issues on other types of Boolean networks would also be of interest.
}

\bibliographystyle{IEEEtran}
\bibliography{paperrefs[4]}

\begin{thebibliography}{10}
\providecommand{\url}[1]{#1}
\csname url@samestyle\endcsname
\providecommand{\newblock}{\relax}
\providecommand{\bibinfo}[2]{#2}
\providecommand{\BIBentrySTDinterwordspacing}{\spaceskip=0pt\relax}
\providecommand{\BIBentryALTinterwordstretchfactor}{4}
\providecommand{\BIBentryALTinterwordspacing}{\spaceskip=\fontdimen2\font plus
\BIBentryALTinterwordstretchfactor\fontdimen3\font minus
  \fontdimen4\font\relax}
\providecommand{\BIBforeignlanguage}[2]{{%
\expandafter\ifx\csname l@#1\endcsname\relax
\typeout{** WARNING: IEEEtran.bst: No hyphenation pattern has been}%
\typeout{** loaded for the language `#1'. Using the pattern for}%
\typeout{** the default language instead.}%
\else
\language=\csname l@#1\endcsname
\fi
#2}}
\providecommand{\BIBdecl}{\relax}
\BIBdecl

\bibitem{shmulevich2002boolean}
I.~Shmulevich, E.~R. Dougherty, and W.~Zhang, ``From {B}oolean to probabilistic
  {B}oolean networks as models of genetic regulatory networks,''
  \emph{Proceedings of the IEEE}, vol.~90, no.~11, pp. 1778--1792, 2002.

\bibitem{smolen2000mathematical}
P.~Smolen, D.~A. Baxter, and J.~H. Byrne, ``Mathematical modeling of gene
  networks,'' \emph{Neuron}, vol.~26, no.~3, pp. 567--580, 2000.

\bibitem{hasty2001computational}
J.~Hasty, D.~McMillen, F.~Isaacs, and J.~J. Collins, ``Computational studies of
  gene regulatory networks: in numero molecular biology,'' \emph{Nature Reviews
  Genetics}, vol.~2, no.~4, pp. 268--279, 2001.

\bibitem{noual2013non}
M.~Noual, D.~Regnault, and S.~Sen{\'e}, ``About non-monotony in {B}oolean
  automata networks,'' \emph{Theoretical Computer Science}, vol. 504, pp.
  12--25, 2013.

\bibitem{kauffman1969metabolic}
S.~A. Kauffman, ``Metabolic stability and epigenesis in randomly constructed
  genetic nets,'' \emph{Journal of Theoretical Biology}, vol.~22, no.~3, pp.
  437--467, 1969.

\bibitem{kauffman1969homeostasis}
S.~Kauffman, ``Homeostasis and differentiation in random genetic control
  networks,'' \emph{Nature}, vol. 224, pp. 177--178, 1969.

\bibitem{thomas1990biological}
R.~Thomas and R.~D'Ari, \emph{Biological Feedback}.\hskip 1em plus 0.5em minus
  0.4em\relax CRC press, 1990.

\bibitem{mcculloch1943logical}
W.~S. McCulloch and W.~Pitts, ``A logical calculus of the ideas immanent in
  nervous activity,'' \emph{The Bulletin of Mathematical Biophysics}, vol.~5,
  no.~4, pp. 115--133, 1943.

\bibitem{hopfield1982neural}
J.~J. Hopfield, ``Neural networks and physical systems with emergent collective
  computational abilities,'' \emph{Proceedings of the National Academy of
  Sciences}, vol.~79, no.~8, pp. 2554--2558, 1982.

\bibitem{hopfield1984neurons}
------, ``Neurons with graded response have collective computational properties
  like those of two-state neurons,'' \emph{Proceedings of the National Academy
  of Sciences}, vol.~81, no.~10, pp. 3088--3092, 1984.

\bibitem{akutsu1999identification}
T.~Akutsu, S.~Miyano, S.~Kuhara \emph{et~al.}, ``Identification of genetic
  networks from a small number of gene expression patterns under the {B}oolean
  network model.'' in \emph{Pacific Symposium on Biocomputing}, vol.~4.\hskip
  1em plus 0.5em minus 0.4em\relax Citeseer, 1999, pp. 17--28.

\bibitem{davidich2008boolean}
M.~I. Davidich and S.~Bornholdt, ``{B}oolean network model predicts cell cycle
  sequence of fission yeast,'' \emph{PLOS ONE}, vol.~3, no.~2, p. e1672, 2008.

\bibitem{thomas1973boolean}
R.~Thomas, ``Boolean formalization of genetic control circuits,'' \emph{Journal
  of Theoretical Biology}, vol.~42, no.~3, pp. 563--585, 1973.

\bibitem{raeymaekers2002dynamics}
L.~Raeymaekers, ``Dynamics of {B}oolean networks controlled by biologically
  meaningful functions,'' \emph{Journal of Theoretical Biology}, vol. 218,
  no.~3, pp. 331--341, 2002.

\bibitem{sontag2008effect}
E.~Sontag, A.~Veliz-Cuba, R.~Laubenbacher, and A.~S. Jarrah, ``The effect of
  negative feedback loops on the dynamics of {B}oolean networks,''
  \emph{Biophysical Journal}, vol.~95, no.~2, pp. 518--526, 2008.

\bibitem{colon2005boolean}
O.~Col{\'o}n-Reyes, R.~Laubenbacher, and B.~Pareigis, ``Boolean monomial
  dynamical systems,'' \emph{Annals of Combinatorics}, vol.~8, no.~4, pp.
  425--439, 2005.

\bibitem{colon2006monomial}
O.~Col{\'o}n-Reyes, A.~Jarrah, R.~Laubenbacher, and B.~Sturmfels, ``Monomial
  dynamical systems over finite fields,'' \emph{arXiv preprint math/0605439},
  2006.

\bibitem{park2014monomial}
J.~Park and S.~Gao, ``Monomial dynamical systems in \# {P}-complete,''
  \emph{Mathematical Journal of Interdisciplinary Sciences}, vol.~1, no.~1,
  2012.

\bibitem{jarrah2010dynamics}
A.~S. Jarrah, R.~Laubenbacher, and A.~Veliz-Cuba, ``The dynamics of conjunctive
  and disjunctive {B}oolean network models,'' \emph{Bulletin of Mathematical
  Biology}, vol.~72, no.~6, pp. 1425--1447, 2010.

\bibitem{aracena2004limit}
J.~Aracena, J.~Demongeot, and E.~Goles, ``On limit cycles of monotone functions
  with symmetric connection graph,'' \emph{Theoretical Computer Science}, vol.
  322, no.~2, pp. 237--244, 2004.

\bibitem{zhao2005remark}
Q.~Zhao, ``A remark on `scalar equations for synchronous {B}oolean networks
  with biological applications' by {C}. {F}arrow, {J}. {H}eidel, {J}.
  {M}aloney, and {J}. {R}ogers,'' \emph{IEEE Transactions on Neural Networks},
  vol.~16, no.~6, pp. 1715--1716, 2005.

\bibitem{melliti2016asynchronous}
T.~Melliti, D.~Regnault, A.~Richard, and S.~Sen{\'e}, ``Asynchronous simulation
  of {B}oolean networks by monotone boolean networks,'' in \emph{International
  Conference on Cellular Automata}.\hskip 1em plus 0.5em minus 0.4em\relax
  Springer, 2016, pp. 182--191.

\bibitem{lingas2017towards}
A.~Lingas, ``Towards an almost quadratic lower bound on the monotone circuit
  complexity of the {B}oolean convolution,'' in \emph{International Conference
  on Theory and Applications of Models of Computation}.\hskip 1em plus 0.5em
  minus 0.4em\relax Springer, 2017, pp. 401--411.

\bibitem{translation}
S.~Clancy and W.~Brown, ``Translation: {DNA} to {mRNA} to protein,''
  \emph{Nature Education}, vol.~1, no.~1, p. 101, 2008.

\bibitem{gossen1992tight}
M.~Gossen and H.~Bujard, ``Tight control of gene expression in mammalian cells
  by tetracycline-responsive promoters.'' \emph{Proc. National Academy of
  Sciences}, vol.~89, no.~12, pp. 5547--5551, 1992.

\bibitem{derisi1997exploring}
J.~L. DeRisi, V.~R. Iyer, and P.~O. Brown, ``Exploring the metabolic and
  genetic control of gene expression on a genomic scale,'' \emph{Science}, vol.
  278, no. 5338, pp. 680--686, 1997.

\bibitem{stragier1988processing}
P.~Stragier, C.~Bonamy, and C.~Karmazyn-Campelli, ``Processing of a sporulation
  sigma factor in bacillus subtilis: how morphological structure could control
  gene expression,'' \emph{Cell}, vol.~52, no.~5, pp. 697--704, 1988.

\bibitem{helene1991anti}
C.~Helene, ``The anti-gene strategy: control of gene expression by
  triplex-forming-oligonucleotides.'' \emph{Anti-cancer Drug Design}, vol.~6,
  no.~6, pp. 569--584, 1991.

\bibitem{menolascina2014vivo}
F.~Menolascina, G.~Fiore, E.~Orabona, L.~De~Stefano, M.~Ferry, J.~Hasty,
  M.~di~Bernardo, and D.~di~Bernardo, ``In-vivo real-time control of protein
  expression from endogenous and synthetic gene networks,'' \emph{PLOS Comput
  Biol}, vol.~10, no.~5, p. e1003625, 2014.

\bibitem{milias2011silico}
A.~Milias-Argeitis, S.~Summers, J.~Stewart-Ornstein, I.~Zuleta, D.~Pincus,
  H.~El-Samad, M.~Khammash, and J.~Lygeros, ``In silico feedback for in vivo
  regulation of a gene expression circuit,'' \emph{Nature Biotechnology},
  vol.~29, no.~12, pp. 1114--1116, 2011.

\bibitem{pathak2013optogenetic}
G.~P. Pathak, J.~D. Vrana, and C.~L. Tucker, ``Optogenetic control of cell
  function using engineered photoreceptors,'' \emph{Biology of the Cell}, vol.
  105, no.~2, pp. 59--72, 2013.

\bibitem{uhlendorf2012long}
J.~Uhlendorf, A.~Miermont, T.~Delaveau, G.~Charvin, F.~Fages, S.~Bottani,
  G.~Batt, and P.~Hersen, ``Long-term model predictive control of gene
  expression at the population and single-cell levels,'' \emph{Proceedings of
  the National Academy of Sciences}, vol. 109, no.~35, pp. 14\,271--14\,276,
  2012.

\bibitem{cheng2009controllability}
D.~Cheng and H.~Qi, ``Controllability and observability of {B}oolean control
  networks,'' \emph{Automatica}, vol.~45, no.~7, pp. 1659--1667, 2009.

\bibitem{zhang2012controllability}
L.~Zhang, J.~Feng, and J.~Yao, ``Controllability and observability of switched
  {B}oolean control networks,'' \emph{IET Control Theory \& Applications},
  vol.~6, no.~16, pp. 2477--2484, 2012.

\bibitem{liu2014some}
Y.~Liu, J.~Lu, and B.~Wu, ``Some necessary and sufficient conditions for the
  output controllability of temporal {B}oolean control networks,'' \emph{ESAIM:
  Control, Optimisation and Calculus of Variations}, vol.~20, no.~1, pp.
  158--173, 2014.

\bibitem{liu2016state}
R.~Liu, C.~Qian, S.~Liu, and Y.-F. Jin, ``State feedback control design for
  {B}oolean networks,'' \emph{BMC Systems Biology}, vol.~10, no.~3, p.~70,
  2016.

\bibitem{luo2016controllability}
C.~Luo, X.~Zhang, R.~Shao, and Y.~Zheng, ``Controllability of {B}oolean
  networks via input controls under {H}arvey's update scheme,'' \emph{Chaos: An
  Interdisciplinary Journal of Nonlinear Science}, vol.~26, no.~2, p. 023111,
  2016.

\bibitem{cheng2010analysis}
D.~Cheng, H.~Qi, and Z.~Li, \emph{Analysis and Control of {B}oolean Networks: A
  Semi-tensor Product Approach}.\hskip 1em plus 0.5em minus 0.4em\relax
  Springer Science \& Business Media, 2010.

\bibitem{laschov2012controllability}
D.~Laschov and M.~Margaliot, ``Controllability of {B}oolean control networks
  via the {P}erron--{F}robenius theory,'' \emph{Automatica}, vol.~48, no.~6,
  pp. 1218--1223, 2012.

\bibitem{li2015controllability}
R.~Li, M.~Yang, and T.~Chu, ``Controllability and observability of {B}oolean
  networks arising from biology,'' \emph{Chaos: An Interdisciplinary Journal of
  Nonlinear Science}, vol.~25, no.~2, p. 023104, 2015.

\bibitem{orbitcontrol}
Z.~Gao, X.~Chen, and T.~Ba{\c{s}}ar, ``Orbit-controlling sets for conjunctive
  {B}oolean networks,'' in \emph{Proc. 2017 American Control Conference (ACC)},
  Seattle, WA, May 24--26, 2017, pp. 4989--4994.

\bibitem{statecontrol}
------, ``State-controlling sets for conjunctive {B}oolean networks,'' in
  \emph{Proc. 20th IFAC World Congress}, Toulouse, France, Jul. 9--14, 2017,
  pp. 14\,855--14\,860.

\bibitem{stabilityfull}
------, ``Stability structures of conjunctive {B}oolean networks,'' submitted
  to \emph{Automatica}, available on \emph{arXiv preprint arXiv:1603.04415},
  2016.

\bibitem{goles2012disjunctive}
E.~Goles and M.~Noual, ``Disjunctive networks and update schedules,''
  \emph{Advances in Applied Mathematics}, vol.~48, no.~5, pp. 646--662, 2012.

\bibitem{ruz2013preservation}
G.~A. Ruz, M.~Montalva, and E.~Goles, ``On the preservation of limit cycles in
  {B}oolean networks under different updating schemes,'' \emph{Advances in
  Artificial Life, ECAL}, pp. 1085--1090, 2013.

\bibitem{aracena2016number}
J.~Aracena, A.~Richard, and L.~Salinas, ``Number of fixed points and disjoint
  cycles in monotone boolean networks,'' \emph{arXiv preprint
  arXiv:1602.03109}, 2016.

\bibitem{robert2012discrete}
F.~Robert, \emph{Discrete iterations: a metric study}.\hskip 1em plus 0.5em
  minus 0.4em\relax Springer Science \& Business Media, 2012, vol.~6.

\bibitem{gao2016periodic}
Z.~Gao, X.~Chen, J.~Liu, and T.~Ba{\c{s}}ar, ``Periodic behavior of a diffusion
  model over directed graphs,'' in \emph{Proc. 55th Conference on Decision and
  Control (CDC)}, Las Vegas, NV, Dec. 12--14, 2017, pp. 37--42.

\bibitem{chen2017asymptotic}
X.~Chen, Z.~Gao, and T.~Ba\c{s}ar, ``Asymptotic behavior of conjunctive
  {B}oolean networks over weakly connected digraphs,'' in \emph{Proc. 56th
  Conf. on Decision and Control (CDC)}, Melbourne, Australia, 2017, to appear.

\bibitem{karp1972reducibility}
R.~M. Karp, ``Reducibility among combinatorial problems,'' in \emph{Complexity
  of Computer Computations}.\hskip 1em plus 0.5em minus 0.4em\relax Springer,
  1972, pp. 85--103.

\bibitem{fomin_et_al:LIPIcs:2010:2470}
F.~V. Fomin and Y.~Villanger, ``{Finding Induced Subgraphs via Minimal
  Triangulations},'' in \emph{27th Internat. Symp. on Theoretical Aspects of
  Computer Science}, ser. Leibniz Internat. Proc. Informatics (LIPIcs), vol.~5,
  2010, pp. 383--394.

\bibitem{razgon2007computing}
I.~Razgon, ``Computing minimum directed feedback vertex set in
  $\mathcal{O}(1.9977^n)$.'' in \emph{ICTCS}, 2007, pp. 70--81.

\bibitem{chen2011cluster}
Y.~Chen, J.~L{\"u}, F.~Han, and X.~Yu, ``On the cluster consensus of
  discrete-time multi-agent systems,'' \emph{Systems \& Control Letters},
  vol.~60, no.~7, pp. 517--523, 2011.

\bibitem{han2013cluster}
Y.~Han, W.~Lu, and T.~Chen, ``Cluster consensus in discrete-time networks of
  multiagents with inter-cluster nonidentical inputs,'' \emph{IEEE Transactions
  on Neural Networks and Learning Systems}, vol.~24, no.~4, pp. 566--578, 2013.

\end{thebibliography}

\section*{Appendix}

This Appendix is organized into two subsections. In the first subsection, we provide some preliminary results that are necessary for proving Proposition~\ref{initial}. In the second subsection, we provide the analysis and proof for Proposition~\ref{initial}.

\subsection{Irreducible components of strongly connected graphs}\label{borrow}
We introduce here a tool that we built in~\cite{stabilityfull}: decomposing the dependency graph into several irreducible components. Similar decompositions have also been studied in~\cite{chen2011cluster,han2013cluster}. Proofs of these results can be found in~\cite{stabilityfull}. 


\subsubsection{Irreducible components}\label{irrcomp}

In this sub-subsection, we first construct $p^*$ digraphs, as we call the irreducible components of $D$. Then, we define a CBN, as we call an induced dynamics, on each irreducible component. We further present the relationships between the original dynamics and the $p^*$ induced dynamics.  To proceed, we introduce some definitions.

\begin{definition}\label{rel}
	Let $p$ divide the lengths of cycles of the dependency graph~$D$. We say that a vertex $v_i$ is {\bf related to} $v_j$ (or simply write $v_i\sim_p v_j$) if there exists a walk $w_{ij}$ from $v_i$ to $v_j$ such that $p$ divides $l(w_{ij})$.
\end{definition}
We note here that the relation introduced in Definition~\ref{rel} is in fact an \emph{equivalence relation}. We then construct  a subset of $V$ as follows: First, choose an arbitrary vertex $v_i$ as a base vertex; then, define
\begin{equation}\label{eq:defWpi}
[v_i]_p:=\{v_j\in V\mid v_j\sim_p v_i\}.
\end{equation}

Note that 
the subset $[v_i]_p$, for any $v_i\in V$, is an \emph{equivalence class} of $v_i$. 
We further have the following result: 

\begin{definition}
	Let $D = (V,E)$ be the dependency graph associated with a CBN. The digraph $D$ is said to be {\bf irreducible} if $p^*=1$.
\end{definition}
If the digraph $D$ is not irreducible, then there is a decomposition of $D$ into $p^*$ components each of which is irreducible~\cite{stabilityfull}. 
This decomposition can be described as follows: First, picking an arbitrary vertex $v_0$, we obtain a subset $[v_0]_{p^*}$  via~\eqref{eq:defWpi}. For ease of notation, we will write $[v_0]$ instead of $[v_0]_{p}$ if  $p = p^*$. 
Now, picking vertices $v_1\in \N(v_0),\ldots, v_{p^*-1} \in \N(v_{p^*-2})$, we obtain subsets $[v_1],\ldots, [v_{p^*-1}]$. It turns out that these subsets form a partition of $V$~\cite{stabilityfull}. An example of such a partition is provided in Fig.~\ref{P}. 
\begin{figure}[h]
	\centering
	\includegraphics[height=60mm]{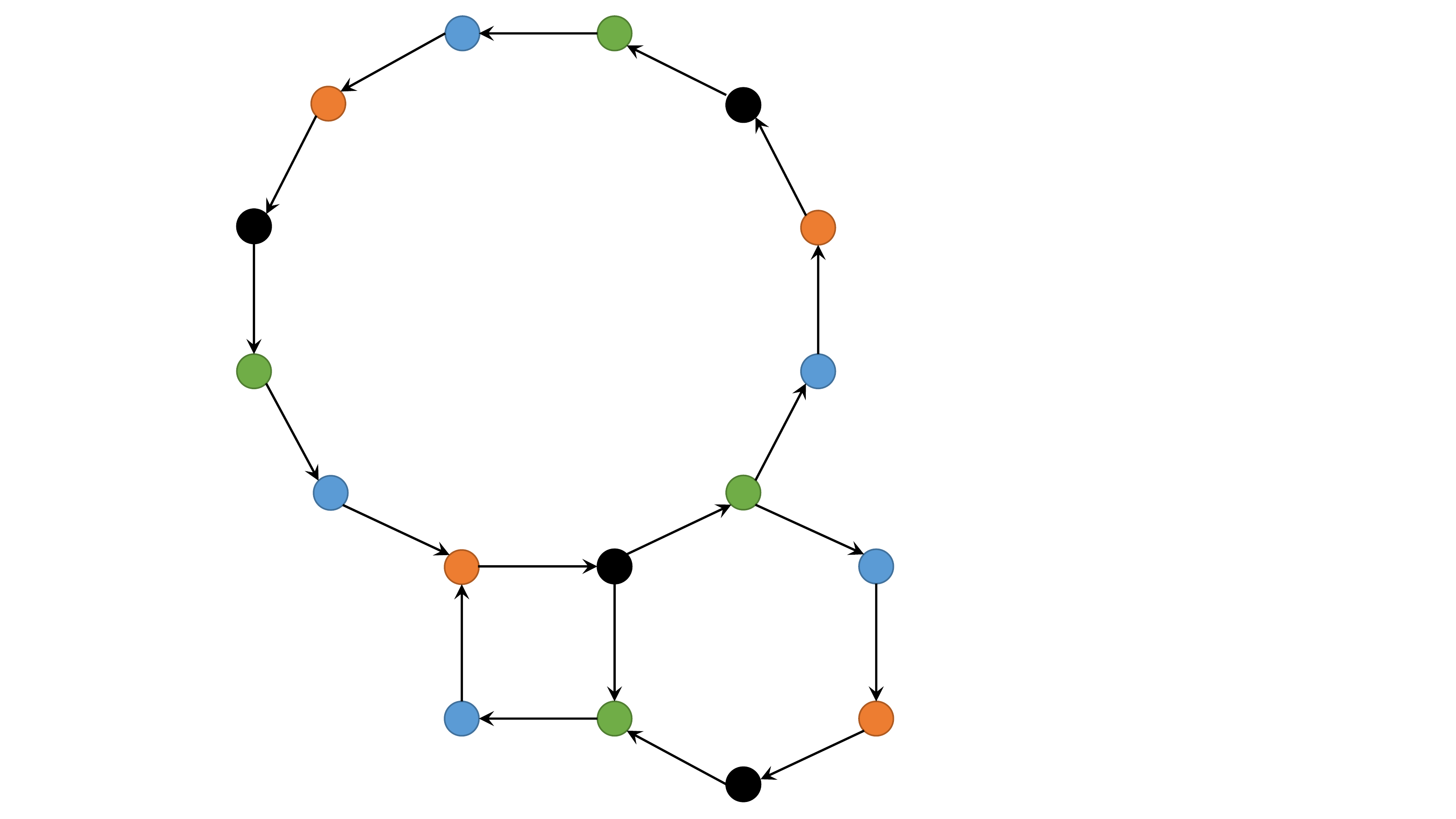}
	\caption{The digraph in the figure has three cycles whose lengths are $4$,$8$, and $12$, respectively. Let $p = 4$ be a common divisor of the cycle lengths. Then, the associated partition yields $4$ disjoint subsets, with the vertices of the same color belong to the same subset.  
	}
	\label{P}
\end{figure}

We then have the following definition.

\begin{definition}[Irreducible components]\label{Def:IrrComp}
	Let $D = (V, E)$ be a strongly connected digraph, and $p^*$ be its loop number. Let the subsets  $[v_0],\ldots, [v_{p^*-1}]$ form a partition of $V$. The {\bf irreducible components} of $D$ are digraphs {\color{black} $G_0 = (U_0,F_0),\ldots, G_{p^*-1} = (U_{p^*-1}, F_{p^*-1})$,} with their vertex sets $U_k$'s given by
	$$
	U_k := [v_k], \hspace{10pt} \forall k = 0,\ldots, p^*-1.
	$$
	The edge set $F_k$ of $G_k$ is determined as follows: Let $u_i$ and $u_j$ be two vertices of $G_k$. Then, $u_iu_j$ is an edge of $G_k$ if there is a walk $w_{ij}$ from $u_i$ to $u_j$ in $D$ with $l(w_{ij}) = p^*$. 
\end{definition}

We provide an example in Fig.~\ref{components} in which we show the irreducible components of the digraph shown in Fig.~\ref{P}.

\begin{figure}[h]
	\centering
	\includegraphics[height=60mm]{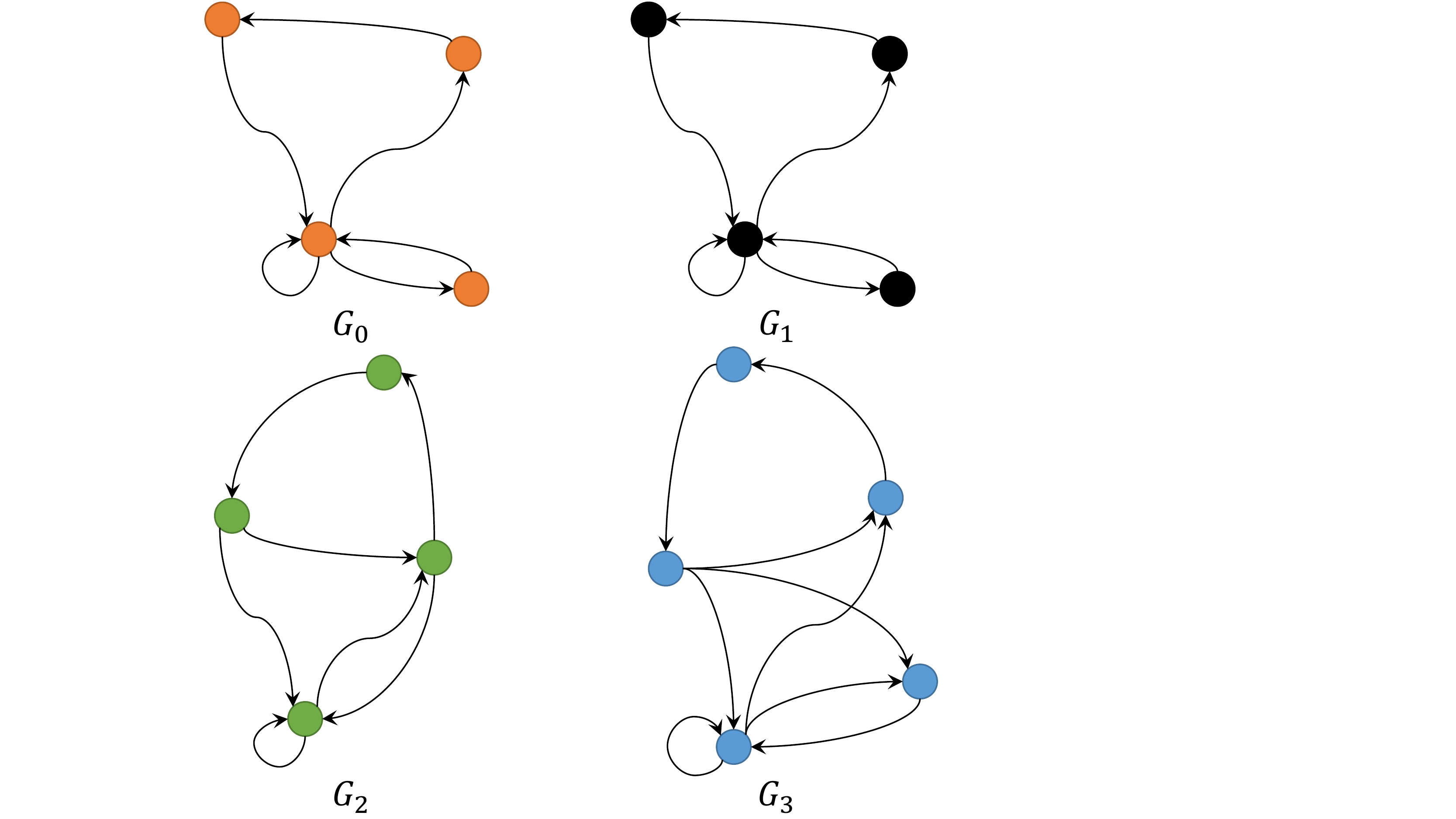}
	\caption{Irreducible components of the digraph shown in Fig.~\ref{P}.}
	\label{components}
\end{figure}

It can be shown that each irreducible component $G_k$, $k=0,\ldots,p^*-1$, is strongly connected and irreducible~\cite{stabilityfull}.

Given a subset $V'$ of $V$ and a nonnegative integer $p$, we define a subset $\Nin^p(V')$ by induction: For $p = 0$, let $\Nin^0(V') := V'$; for $p \ge 1$, we define 
\begin{equation}\label{eq:defninp}
\Nin^p(V') := \cup_{v_j\in \Nin^{p-1}(V')} \Nin(v_j). 
\end{equation}
Similarly, we define $\N^p(V')$ by replacing $\Nin$ with $\N$ in~\eqref{eq:defninp}. With the notations above, we have the following result about the relationships between the vertex sets of the irreducible components:

\begin{lemma}\label{pro:NinNout}
	For $k\ge 0$, we have 
	$$
	\left\{
	\begin{array}{l}
	\N^k(U_0) = U_{(k \bmod p^*)}, \vspace{3pt}\\
	\Nin^k(U_0) = U_{(-k \bmod p^*)}.
	\end{array}
	\right.
	$$\,
\end{lemma}

\subsubsection{Induced dynamics}\label{def:indyn}
Let $f = (f_1,\ldots, f_n)$ be a CBN, and $D$ be the dependency graph. Let $G_0,\ldots, G_{p^*-1}$ be the irreducible components of $D$. Now, for each $k = 0,\ldots, p^*-1$, we can define a CBN as follows:

\begin{definition}[Induced dynamics]
	An {\bf induced dynamics} on $G_k$ is a CBN whose dependency graph is $G_k$. 
\end{definition}

We can express the induced dynamics on $G_k$ explicitly as follows: Let $U_k = \{u_1,\ldots, u_m\}$, and $(y_1,\ldots, y_m)$ be the state of the network. Let $g_k= (g_{k_1},\ldots,g_{k_m})$ be the associated value update rule. Then,
$$
g_{k_i}(y_1,\ldots, y_m) = \prod_{u_j\in U_k} y^{\epsilon_{ji}}_j 
$$    
where $\epsilon_{ji} = 1$ if $u_j$ is an in-neighbor of $u_i$ and $\epsilon_{ji} = 0$ otherwise.

We now relate the original dynamics $f$ on $D$ to the induced dynamics on the irreducible components.  We first introduce some notations. Let $V'$ be a subset of $V$. We define $f_{V'}$ to be the restriction of $f$ to $V'$. For a positive integer~$p$, we let $f^p$ be the map defined by applying the map $f$ for $p$ times.
We now introduce the following result:

\begin{proposition}\label{dynamic}
	Let $G_k = (U_k, F_k)$ be an irreducible component of $D$. Then, the following hold: 
	\begin{enumerate}
		\item Let $g_{k}$ be the induced dynamics on $G_k$. Then, 
		$$
		g_k(x_{U_k}) = f^{p^*}_{U_k}(x), \hspace{10pt} \forall x \in \F^n_2. 
		$$
		\item Suppose that $x(t_0)$ is in a periodic orbit; then,
		\begin{equation}\label{eq:valueshift}
		x_{U_{(k+1 \bmod p^*)}}(t_0 + 1) = x_{U_{k}}(t_0)
		\end{equation}
	\end{enumerate}\,
\end{proposition}

We note here that if $x(t_0)$ is in a periodic orbit, then for each $k = 0,\ldots, p^*-1$, the entries of $x_{U_k}(t_0)$ hold the same value. This indeed follows from the first item of Proposition~\ref{dynamic}:

\begin{corollary}\label{same}
	Let $D = (V, E)$ be the dependency graph of a CBN, and  $G_k = (U_k, F_k)$, for $k = 0,\ldots,p^*-1$, be its irreducible components. 
	A state $x\in \F^n_2$ is in a periodic orbit of the CBN if and only if for each $k=0,\ldots,p^*-1$, the entries of $x_{U_k}$ hold the same value.
\end{corollary}

So, if $x(t_0)$ is in a periodic orbit, then from the second item of Proposition~\ref{dynamic} and Corollary~\ref{same}, the entries of $x_{U_k}(t_0)$ hold the same value, and moreover, this value will be passed onto the entries of $x_{U_{(k+1 \bmod p^*)}}$ at the next time step. 

\subsection{Analysis and Proof for Proposition~\ref{initial}}

It should be clear that after executing the ``while" loop of Algorithm~\ref{orbitalgo}, the state of the system is given by $x(\tau-1)=(1,\ldots,1)$. Then, by assigning $y_0$ to $x_i$ at time $\tau$, we have $x_i(\tau)=y$ and $x_j(\tau)=1$ for all $v_j\neq v_i$. We first have the following lemma.

\begin{lemma}\label{process}
	Let $D=(V,E)$ be the dependency graph of a CBN. Let $v_i\in V$ be arbitrary, and without loss of generality, assume that $v_i\in U_0$. Let the initial condition be $x_i(0)=y$ and $x_j(0)=1$ for all $v_j\in U_0$. Then, for $t'=0,\ldots,p^*-1$, we have 
	\begin{eqnarray}\label{y0}
	x_{\N^{t'}(v_i)}(t')=y{\bf 1}.
	\end{eqnarray}
\end{lemma}
\begin{proof}
	The proof is carried out by induction on $t'$. For the base case $t'=0$, it is true since $\N^{0}(v_i)=\{v_i\}$, and hence $x_{\N^{0}(v_i)}(0)=x_i(0)=y$ by assumption. For the induction step, we assume that~\eqref{y0} holds for $t'=k$, where $0\leq k<p^*-1$; then we show that~\eqref{y0} holds for $t'=k+1$.
	
	Let $v_a$ be an arbitrary vertex in $\N^{k+1}(v_i)$, and $v_b\in \Nin(v_a)\cap \N^k(v_i)$. Then, by Lemma~\ref{pro:NinNout}, $v_a\in U_{k+1}$ and $v_b\in U_k$. Thus, $\Nin(v_a)\subseteq U_k$.  By induction assumption, $x_b(k)=y$. If $y=0$, then $x_a(k+1)=x_b(k)=0=y$. If $y=1$, then $x_{U}(0)={\bf 1}$ by assumption. Again from Lemma~\ref{pro:NinNout}, $\Nin(U_1)=U_0,\Nin(U_2)=U_1,\ldots,\Nin(U_{k+1})=\Nin(U_k)$. Thus, $x_{U_{k+1}}(k+1)=x_{U_k}(k)=\ldots=x_{U_0}(0)={\bf 1}$. This leads to $x_a(k+1)=1=y$.
\end{proof}

Before proceeding to the proof of Proposition~\ref{initial}, we need to revisit the fact we stated in Lemma~\ref{bijec}. In Lemma~\ref{bijec}, we have shown that there is a bijection between the set of periodic orbits and the set of binary necklaces of length $p^*$. Now, with the graph decomposition that we introduced in the previous subsection, one can describe the bijection map as follows: 
First, in Corollary~\ref{same}, we have shown that a state $x\in \mathbb{F}^n_2$ is in a periodic orbit if and only if for each $k=0,\ldots,p^*-1$, the entries of $x_{U_k}$ hold the same value.  Therefore, we represent this periodic orbit as a binary necklace $s=y_0\ldots y_{p^*-1}$, by taking the value of the entries of $x_{U_{k}}$ as $y_k$, for all $k=0,\ldots,p^*-1$. 

With the above fact and Lemma~\ref{process} at hand, we now prove Proposition~\ref{initial}.

\begin{proof}[Proof of Proposition~\ref{initial}]
	We first show that the state $x$ at time $(\tau+p^*-1)$ is given by~\eqref{beginfor}.
	
	Without loss of generality, assume that $v_i\in U_0$. Then, by assigning $y_{p^*-1}$ to $x_i$ at time $\tau$, we have $x_i(\tau)=y_{p^*-1}$ and $x_j(\tau)=1$ for all $v_j\neq v_i$. Then, by applying Lemma~\ref{process} with $t'=p^*-1$, we obtain that $x_{\N^{p^*-1}(v_i)}(T+p^*-1)=y_{p^*-1}{\bf 1}$. At time $\tau+1$, we are assigning $x_i(\tau+1)=y_{p^*-2}$. Note that $U_0,\ldots,U_{p^*-1}$ are pairwise distinct since they form a partition of $V$. Thus, $x_{U_0}(\tau+1)=x_{U_{p^*-1}}(\tau)={\bf 1}$. We can then apply Lemma~\ref{process} again with $t'=p^*-2$ to obtain that $x_{\N^{p^*-2}(v_i)}(\tau+p^*-1)=y_{p^*-2}{\bf 1}$. Continuing on this pattern, we will finally obtain that $x_{\N^{j}(v_i)}(\tau+p^*-1)=y_{j}{\bf 1}$ for all $j=0,\ldots,p^*-1$. Any vertices not reached by the assigned values at time $\tau+p^*-1$ still hold ``$1$". Thus, the state $x$ at time $(\tau+p^*-1)$ is given by~\eqref{beginfor}.

	We then show that the system~\eqref{eq:updaterule}, with~\eqref{beginfor} being the initial condition, will enter the periodic orbit $s$.
	Without loss of generality, assume that $v_i\in U_0$; then $\N^j(v_i)\subseteq U_{j}$ for $j=0,\ldots,p^*-1$. 
	
	If $y_j=0$, then $x_{U_j}(0)$ contains an entry of value $0$. Consider the induced dynamics on $G_j$:  First, from the value update rule and the first item of Proposition~\ref{dynamic}, if $x_{U_j}(0)$ contains an entry of value $0$, then so does $x_{U_j}(tp^*)$ for all $t \ge 0$. Second, since $G_0$ is irreducible, a periodic orbit of the induced dynamics has to be a fixed point~\cite{jarrah2010dynamics,gao2016periodic}. Combining these two facts, we know that there is a time $t_0\ge 0$ such that $x_{U_j}(tp^*) = {\bf 0}$ for all $t \ge t_0$. 
	
	If $y_j=1$, then $x_{U_j}(0)={\bf 1}$. We appeal again to the first item of Theorem~\ref{dynamic} and obtain
	$$
	x_{U_j}(tp^*) = f^{tp^*}_{U_j}(x_{U_j}(0)) = g^t_j(x'_{U_j}(0)) = x_{U_j}(0)={\bf 1}.
	$$
	
	Therefore, we conclude that $x_{U_j}(t_0p^*)=y_j{\bf 1}$, and this holds for all $j=0,\ldots,p^*-1$. The system is thus in periodic orbit $s=y_0\ldots y_{p^*-1}$.
\end{proof}

\end{document}